\documentclass[reqno]{amsart}

\usepackage{animate} 
\usepackage{graphicx}
\usepackage{mathtools} 
\usepackage{float} 
\usepackage{color} 
\usepackage{subcaption} 
\usepackage[utf8]{inputenc} 
\usepackage[margin=3cm]{geometry} 
\usepackage{hyperref}

\newcommand{\RR}{\mathbb R}
\newcommand{\ZZ}{\mathbb Z}
\newcommand{\NN}{\mathbb N}

\newcommand{\const}{\mathrm{const}}
\newcommand{\discr}{\mathrm{discr}}

\newcommand{\loc}{\mathrm{loc}}

\newcommand{\ds}{\, ds}
\newcommand{\dt}{\, dt}
\newcommand{\dtildes}{\, d\tilde s}
\newcommand{\dH}{\, d\mathcal{H}^1}

\DeclareMathOperator*{\argmin}{argmin}

\DeclareMathOperator{\supp}{supp}
\DeclareMathOperator{\Ha}{\mathcal H^1}
\DeclareMathOperator{\der}{d}

\DeclareMathOperator{\Ait}{\mathit A}
\DeclareMathOperator{\Bit}{\mathit B}

\DeclareMathOperator{\Fit}{\mathit F}
\DeclareMathOperator{\Lit}{\mathit L}
\DeclareMathOperator{\Pit}{\mathit P}
\DeclareMathOperator{\Vit}{\mathit V}

\DeclareMathOperator{\cit}{\mathit c}
\DeclareMathOperator{\fit}{\mathit f}

\DeclareMathOperator{\ogamma}{\gamma}
\DeclareMathOperator{\oeta}{\eta}
\DeclareMathOperator{\omu}{\mu}
\DeclareMathOperator{\okappa}{\kappa}
\DeclareMathOperator{\ophi}{\varphi}
\DeclareMathOperator{\opsi}{\psi}

\DeclareMathOperator{\C}{C}
\DeclareMathOperator{\D}{D}
\DeclareMathOperator{\E}{E}
\DeclareMathOperator{\F}{F}
\let\L\relax\DeclareMathOperator{\L}{L}
\DeclareMathOperator{\W}{W}

\DeclareMathOperator{\Rcal}{\mathcal R}

\DeclarePairedDelimiterX\braket[2]{\langle}{\rangle}{#1 \mid #2}
\DeclarePairedDelimiter\up{\lceil}{\rceil}
\DeclarePairedDelimiter\down{\lfloor}{\rfloor}
\DeclarePairedDelimiter\abs{\lvert}{\rvert}
\DeclarePairedDelimiter\norm{\lVert}{\rVert}

\theoremstyle{plain}
	\newtheorem{theorem}{Theorem}
	\newtheorem{corollary}{Corollary}
	
	\newtheorem{lemma}{Lemma}
\theoremstyle{definition}
	\newtheorem{definition}{Definition}
\theoremstyle{remark}
	\newtheorem{remark}{Remark}

\newcommand{\ui}{I}
\newcommand{\AC}{\mathcal{AC}}
\newcommand{\con}{\cit}
\newcommand{\Rot}{\Rcal}
\newcommand{\LM}{\Pit}
\newcommand{\curve}{\ogamma}
\newcommand{\curva}{\okappa}
\newcommand{\len}{\Lit}
\newcommand{\velo}{\Vit}
\newcommand{\dcurve}{\curve^\tau}
\newcommand{\dcurva}{\curva^\tau}
\newcommand{\dlen}{\len^\tau}
\newcommand{\dvelo}{\velo^\tau}
\newcommand{\ddcurve}{\Delta\hat\dcurve}

\DeclareMathOperator{\Err}{Err}

\title[Curve-shortening of open, elastic curves with repelling endpoints]{Curve-shortening flow of open, elastic curves in $\RR^2$ with repelling endpoints: \\ A minimizing movement approach}
\author[R. Badal]
{R. Badal}
\address[Rufat Badal]{Zentrum Mathematik - M7, Technische Universit\"at M\"unchen, Boltzmannstrasse 3, 85748 Garching, Germany}

\begin{document}

\begin{abstract}
    We study an $L^{2}$-type gradient flow of an immersed elastic curve in $\RR^{2}$ whose endpoints repel each other via a Coulomb potential. By De Giorgi's minimizing movements scheme we prove long-time existence of the flow. The work is complemented by several numerical experiments.
\end{abstract}

\maketitle

\tableofcontents

\section{Introduction}
In this paper we study an $L^2$-gradient flow of an open, immersed curve $\curve$ in $\RR^2$ belonging to the following set of admissible curves:
\begin{equation*}
	\AC' := \{ \curve \in H^2(0, 1; \RR^2) \colon \curve_s \neq 0, \, \curve(0) \neq \curve(1) \},
\end{equation*}
where $s$ denotes the curve parameter and $\curve_s := \frac{d}{ds} \curve$ is the speed of the curve. The evolution of a curve $\curve$ is driven by the energy
\begin{equation}\label{defener1}
	\E(\curve) := \L(\curve) + \varepsilon \W(\curve) - \log \abs{\curve(0) - \curve(1)},
\end{equation}
where $\varepsilon > 0$ is a fixed scalar, $\L$ is the length functional and $\W$ is the Willmore-energy defined as
\begin{equation*}
	\W(\curve) = \frac{1}{2} \int_\gamma \curva_\gamma^2 \dH.
\end{equation*}
Here $\curva_\gamma$ denotes the curvature of $\curve$. As the energy $\E$ is invariant under reparameterizations we restrict the class of admissible curves to the following nonlinear subset of $\AC'$:
\begin{equation} \label{defac1}
	\AC = \{ \curve \in \AC' \colon \abs{\curve_s} = \const \}.
\end{equation}

The interest in the gradient flow of functionals as in \eqref{defener1} is motivated by the observation that they represent one of the main energy contributions in several physical systems driven by the formation of topological and geometrical defects, that can be seen, roughly speaking, as codimension two and one singularities of some ad-hoc chosen order parameter, respectively. 

In particular, among the models characterised by the emergence of topological singularities, of particular interest in our case are those featuring fractional vortices. They have been widely studied in the theory of spin systems as a generalization of the classical xy model studied in \cite{alicandro2009variational}, \cite{alicandro2011variational}, \cite{alicandro2014metastability} and \cite{ponsiglione2007elastic} and of superconductors systems as a generalization of the Ginzburg-Landau model for which we refer the reader to \cite{alicandro2014ginzburg}, \cite{bethuel1994ginzburg}, \cite{kosterlitz1973ordering} and \cite{sandier2008vortices}. For what concerns the geometric singularities, they are peculiar of phase separation phenomena in which they represent phase boundaries. Regarding the time evolution of the singularities in this kind of models, we refer the reader to the papers \cite{sandier2004gamma} and \cite{alicandro2016dynamics}.
In \cite{badal2018gamma} the authors study an energy model describing a class of spin systems whose minimizers may develop complicated structures in the form of clusters of phase boundaries possibly connecting fractional vortices (see also \cite{goldman2017ginzburg}), thus providing a first variational analysis of a physical system exhibiting both codimension two and one singularities. Notice that the presence of both type of singularities is considered to be one of the main characteristics of the ground states of physical systems like liquid crystals (see \cite{pang1992string}), in which case the singularities represent disclinations and string defects, or of plastic crystals (see \cite{hull2001introduction}), where they represent partial dislocations and stacking faults. Additionally, they appear also in many micromagnetics and super conductors models (see for instance \cite{alama2006fractional} and \cite{tchernyshyov2005fractional}).

The gradient flow of an energy functional as in \cite{badal2018gamma} turns out to be a very difficult task when considered in its full generality. Keeping the main features of the model, we perform our analysis in the simple case of a line singularity joining two equally charged vortices. In this case the geometric part of the energy which drives the system towards equilibrium takes the form 
\begin{equation*}
	\int_0^1 \abs{\curve_s}_1 \ds - \pi \log \abs{\curve(0) - \curve(1)},
\end{equation*}
where $\abs{\cdot}_1$ denotes the $l^1$-norm. Here $\curve$ parametrizes the line defects with vortices located at $\curve(0)$ and $\curve(1)$. Our energy defined in \eqref{defener1} can be seen as a further simplification of the one above in which we replace the crystalline length by the Euclidean one and we add the Willmore term (thus reducing to an elastica model) whose regularizing effect has been for instance already exploited in \cite{fonseca2012motion} and \cite{piovano2014evolution}. 

One of the main features of the expected flow is the competition between the shortening effect due to the length energy and the end-points repulsion due to the Coulomb potential. As an example (see the end of this introduction for more details) one might consider the simple case of a sufficiently long straight segment. Roughly speaking, if only the Coulomb part would act, the segment would evolve towards an infinite line, while if only the length would be present, it would shorten to a point. Instead, with both terms present, the segment evolves towards a segment having an optimal length which balances the two effects.

In this paper we will model an $L^2$-type gradient flow of the energy $\E$ in \eqref{defener1} employing the De Giorgi's minimizing movement technique described for instance in \cite{ambrosio1995minimizing}. In this case one shows that the flow emerges from a sequence of time-discrete evolutions $(\dcurve)_{\tau > 0}$, where each $\dcurve \colon [0, 1] \times \RR_+ \to \RR^2$ is a piece-wise constant (in time-intervals of length $\tau$) interpolation of a sequence $(\dcurve_n)_n \subset \AC$. This sequence is constructed via a recurrent scheme: Starting from a fixed $\dcurve_0 = \curve_0 \in \AC$ the following curves in the sequence $(\dcurve_n)$ try to decrease $\E$ as much as possible while not straying too far away from the foregoing curve in the sequence. This result is achieved introducing a penalization term $\D$ which can be thought as a dissipation. More precisely we have
\begin{equation} \label{Isaac}
	\left\{
	\begin{aligned}
		\dcurve_{n+1} &\in \argmin_{\curve \in \AC} \left\{ \E(\curve) + \frac{1}{\tau} \D(\curve, \dcurve_n) \right\}, \\
		\dcurve_0 &= \curve_0,
	\end{aligned}
	\right.
\end{equation}
where $\D \colon \AC^2 \to \RR$ is defined as
\begin{equation} \label{Cecilia1}
	\begin{aligned}
		\D(\curve, \tilde \curve) &:= \frac{1}{4 \len_{\tilde\curve}} \int_{0}^{1} \braket{\curve - \tilde\curve}{\Rot(\tilde\curve_s)}^2 \ds + \frac{1}{4 \len_\gamma} \int_{0}^{1} \braket{\curve - \tilde\curve}{\Rot(\curve_s)}^2 \ds \\
		&\quad + \frac{1}{2} \abs{\curve(0) - \tilde\curve(0)}^{2} + \frac{1}{2} \abs{\curve(1) - \tilde\curve(1)}^{2}.
	\end{aligned}
\end{equation}

Even though our model has several common points with that in \cite{fonseca2012motion} and \cite{piovano2014evolution}, it also presents some important differences. In \cite{fonseca2012motion} and \cite{piovano2014evolution} the authors study the morphological evolution of epitaxially strained two-dimensional thin films in terms of the $H^{-1}$ and the $L^2$ gradient flow structure, respectively. On one hand the authors exploit the De Giorgi's minimizing movements technique with curvature regularization, as we do here. On the other hand, they study the time-discrete evolution of the interfaces in the graph setting, that is, assuming that the surfaces are described by a sequence of graphs defined over a fixed coordinate system, whereas we here consider an intrinsic setting. Furthermore their problem does not account for free boundary points as they assume periodic boundary conditions for their graphs. Following their approach in the present case, would complicate our analysis, since we would be forced to consider the motion of graphs on evolving domains of definition. As we consider the $L^2$ gradient flow case as in \cite{piovano2014evolution}, it is also worth comparing the different choices of dissipation. Expressed in intrinsic coordinates, the dissipation in \cite{piovano2014evolution} is
\begin{equation}\label{Piovanodiss}
	\tilde \D(\curve, \tilde \curve) = \frac{1}{2 \len_{\tilde\curve}} \int_{0}^{1} \braket{\curve - \tilde\curve}{\Rot(\tilde\curve_s)}^2 \ds.
\end{equation}
Regarding our dissipation in \eqref{Cecilia1}, besides the presence of additional boundary terms, necessary to control the flow of the free boundary points, we make a different choice of the interior $L^2$ dissipation by considering a symmetrized version of \eqref{Piovanodiss}, namely $\frac{1}{2} \tilde \D(\curve, \tilde \curve) + \frac{1}{2} \tilde \D(\tilde \curve, \curve)$. Such a choice does not change the limit equation (as the two symmetric terms will have the same limit), but it turns out to be convenient from a technical point of view to derive the a priori bounds of the velocity of the time-discrete evolution (see also Lemma \ref{Austin} and Lemma \eqref{Martin}).

\begin{figure}
	\centering
	\includegraphics[width=.6\textwidth]{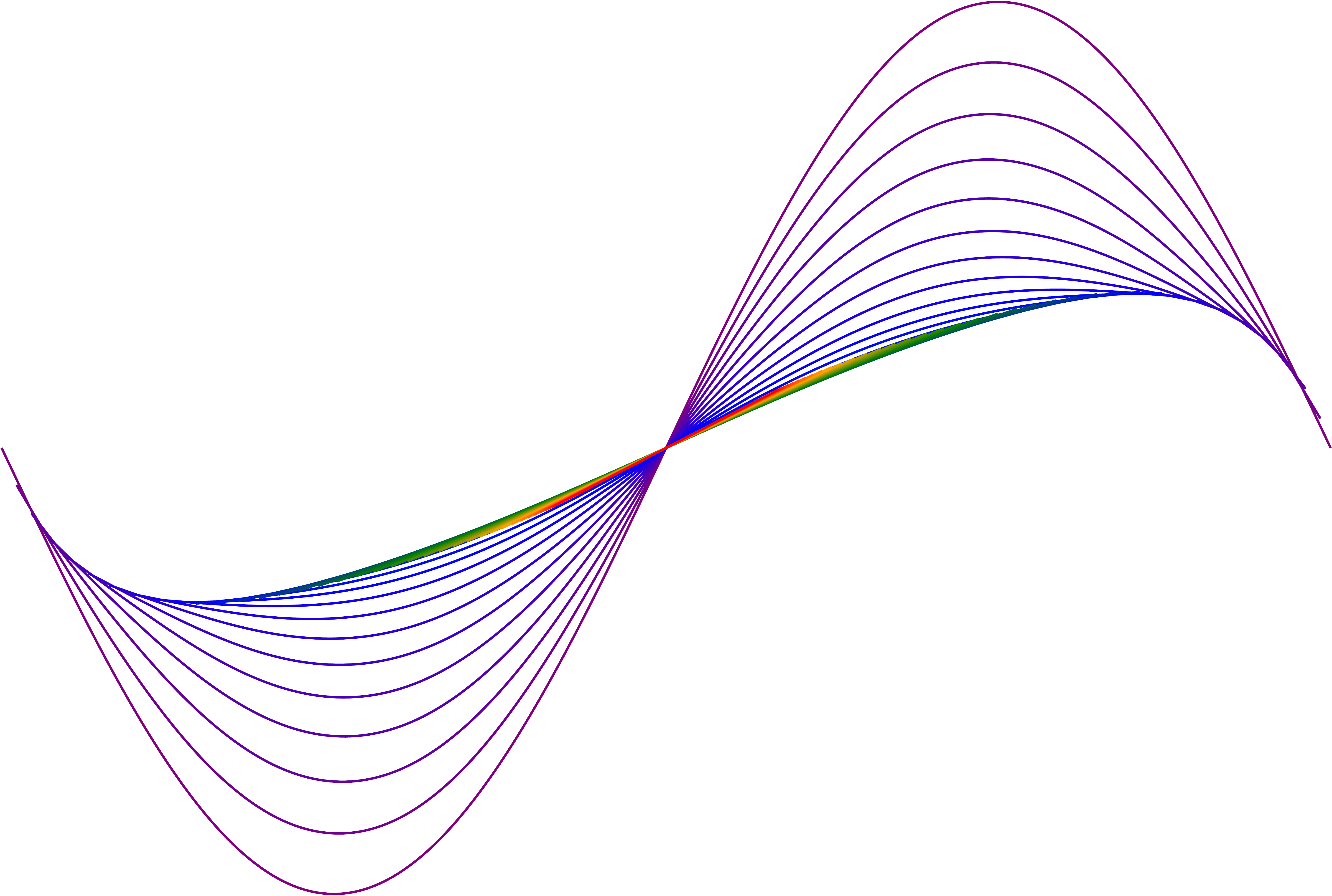}
	\caption{Curve-shortening motion of an initially sinus-shaped curve. The color gradient shows the temporal order of the evolution from violet to red.}
	\label{fig:sin_short}
\end{figure}

We continue by describing the strategy of our existence argument contained in section \ref{Curtis}. Employing the direct methods of the Calculus of Variations we first prove in Theorem \ref{Gordon} the well-posedness of the scheme in \eqref{Isaac}. The Euler-Lagrange equations satisfied at each minimization step in \eqref{Isaac} lead to a weak description of the time-discrete evolution $\dcurve$, see also equation \eqref{Ryan} in Theorem \ref{Betty}. The passage to the limit, as $\tau \to 0$, in the equation governing the time-discrete evolution is eventually obtained in Theorem \ref{Kenneth} by combining Theorems \ref{Emilie}, \ref{Willie}, \ref{Hannah} and \ref{Sally}, where several compactness results for the sequence $\dcurve$ are proved. This part of the argument is closely related to some ideas contained for instance in \cite{piovano2014evolution}. The limit evolution $\curve$ satisfies the initial condition $\curve(0) = \curve_0$ and solves a system of PDEs that are better described through the arclength parameter $\sigma$. With a slight abuse of notation let $\curve(\sigma, t) := \curve(\sigma^{-1} \len(t), t)$, where $\len(t)$ is the length of $\curve$ at time $t$. Then $\curve$ satisfies for almost every $t \in \RR_+$ and for almost every $\sigma \in [0, \len(t)]$ the following system:
\begin{equation} \label{Ricky}
	\left\{
	\begin{aligned}
		\velo^\perp(\sigma, t) &= \curva(\sigma, t) - \varepsilon\left( \curva_{\sigma \sigma}(\sigma, t) + \frac{1}{2} \curva^3(\sigma, t) \right), \\
		\velo^p(t) &= \frac{\curve^p(t) - \curve^q(t)}{\abs{\curve^p(t) - \curve^q(t)}^2} + \curve_\sigma^p(t) - \varepsilon \curva_\sigma^p(t) \Rot(\curve_\sigma^p)(t), \\
		\velo^q(t) &= \frac{\curve^q(t) - \curve^p(t)}{\abs{\curve^q(t) - \curve^p(t)}^2} - \curve_\sigma^q(t) + \varepsilon \curva_\sigma^q(t) \Rot(\curve_\sigma^q)(t), \\
		\curva^p(t) &= \curva^q(t) = 0.
	\end{aligned}
	\right.
\end{equation}
Here $V = \curve_t$ denotes the velocity of the evolution $\curve$, $\curva$ is its curvature, $\Rot$ is the anti-clockwise rotation by $\frac{\pi}{2}$, $\curve_\sigma$ is the unit speed vector of $\curve$ and $\velo^\perp = \braket{V}{\Rot(\curve_\sigma)}$ is the orthogonal component of $V$ with respect to $\curve$. Furthermore the notation $\fit^p(t)$ and $\fit^q(t)$ is a shorthand for $\fit(0, t)$ and $f(\len(t), t)$.

It is worth mentioning that our approach to derive the existence of the limit equation is not the only possible one. There is a vast literature for results concerning the $L^2$-gradient flow of curves or more generally networks driven by elastic energies of Willmore type (see \cite{garcke2017willmore}, \cite{lin2012l2}, \cite{novaga2014curve} and \cite{wen1995curve}). In this setting free boundary points have been considered, too. In contrast to our case, they are usually junction points. This means that the outer points of the network (in our case the end points of the curve) are either fixed (Dirichlet boundary condition) and/or have fixed angles with respect to the boundary of a convex domain containing the network (Neumann boundary condition). For such boundary conditions it is possible to follow a different strategy (see for example \cite{garcke2017willmore}) based on the theory of nonlinear parabolic equations. One main issue in this setting is to guess the right choice of the equation for the tangential component of the speed with respect to the curve which, as in our case, is not given a priori. Such a choice must be done carefully in order to guarantee a well defined system of PDEs. In our case, instead the tangential equation arises naturally from the constant speed constraint in $\AC$. In fact we prove in Theorem \ref{Sally} that $\curve$ satisfies for almost every $t$ and $\sigma$ the following equation
\begin{equation} \label{Estelle}
	\velo^\top_\sigma(\sigma, t) = \frac{\len_t(t)}{\len(t)} + \curva(\sigma, t) \velo^\perp(\sigma, t),
\end{equation}
where $\velo^\top := \braket{V}{\curve_\sigma}$.

In the simple case already mentioned at the beginning of this introduction, of a straight segment $\curve_0$, the system of PDEs in \eqref{Ricky} and \eqref{Estelle} reduces to an ODE for the evolution of the end-points. In fact one can easily prove that the segment remains straight during the evolution and that it  monotonously converges as $t \to \infty$ towards a segment of unit length, which is the global minimizer of $\E$ in \eqref{defener1} and for which the repulsive force of the Coulomb potential and the attractive force of the length term balance. A more interesting example is shown in figure \ref{fig:sin_short}, where we have plotted the step-by-step minimizers defined in \eqref{Isaac} starting with a sinus-shaped $\curve_0$. One can see that the curve starts to straighten out and shorten at the endpoints. As time goes to infinity it also converges asymptotically towards a segment of unit length. We present further numerical experiments in Section \ref{Elmer}.

\section{Notation and Minimizing movement} \label{Curtis}
We will consider planar curves $\curve \colon \ui \to \mathbb{R}^2$, where $\ui := [0, 1]$, whose curve parameter will be denoted by $s \in \ui$.
For derivatives with respect to $s$ we use the index notation, so that for instance $\curve_s$ denotes the first derivative of $\curve$ with respect to $s$. Furthermore, higher order derivatives are written with repeated indices.
The length and the curvature of $\curve$ are denoted by $\len_\gamma$ and $\curva_\gamma$, respectively. Sometimes, when there is no danger of confusion, we will omit the dependence on $\curve$ in the notation.
When there is no confusion in denoting function spaces of such curves we will usually drop the notation for domain and codomain, so that for example $H^2(\ui; \mathbb{R}^2)$ will be abbreviated as $H^2$.
For $t \in R$ we define $\lceil t \rceil$ as the smallest integer $k \in \ZZ$ such that $k \geq t$, and $\lfloor t \rfloor$ as the biggest integer $k \in \ZZ$ such that $k \leq t$.
Given vectors $a, b \in \RR^2$ we denote by $\abs{a}$ its Euclidean length, by $\Rot(a)$ the vector obtained rotating $a$ anti-clockwise by $\frac{\pi}{2}$, and by $\braket{a}{b}$ the Euclidean scalar product between $a$ and $b$.
All over the paper we denote all the constants by $\C$, or $\con$ and we assume that their value is always positive and that it might change from line to line. We moreover explicitly write their dependence on meaningful parameters.

\subsection{Scheme}

We start several objects of relevance to the minimizing movement scheme.
The set of admissible curves is defined as
\begin{equation*}
	\AC := \left\{\curve \in H^{2} \colon \abs{\curve_s} \equiv \mathrm{const} = \len_\gamma, \curve(0) \neq \curve(1)\right\}.
\end{equation*}
Note that for any $\curve \in \AC$ we can derive the following important identity:
\begin{equation} \label{Micheal}
	\curve_{ss}(s) = \len_\gamma \curva_\gamma(s) \Rot(\curve_s)(s) \text{ for a.e. } s \in I.
\end{equation}
Given $\tilde\curve \in \AC$ the subset $\AC_{\tilde\curve} \subset \AC$ is defined as
\begin{equation} \label{Phoebe}
	\AC_{\tilde\curve} := \left\{\curve \in \AC \mid \braket{\curve_s}{\tilde\curve_s} \geq 0\right\}.
\end{equation}
For fixed $\varepsilon >0$ we define the energy $\E \colon \AC \to \mathbb{R}$ as
\begin{equation} \label{Franklin}
	\E(\curve) := \len_\gamma + \frac{\varepsilon}{2} \int_\gamma \curva_\gamma^{2} \der \mathcal{H}^{1} - \log\abs{\curve(1) - \curve(0)}.
\end{equation}
The dissipation $\D \colon \AC^2 \to \mathbb{R}$ is defined as
\begin{equation} \label{Cecilia}
	\begin{aligned}
		\D(\curve, \tilde \curve) &:= \frac{1}{4 \len_{\tilde\curve}} \int_{0}^{1} \braket{\curve - \tilde\curve}{\Rot(\tilde\curve_s)}^2 \ds + \frac{1}{4 \len_\gamma} \int_{0}^{1} \braket{\curve - \tilde\curve}{\Rot(\curve_s)}^2 \ds \\
		&\quad + \frac{1}{2} \abs{\curve(0) - \tilde\curve(0)}^{2} + \frac{1}{2} \abs{\curve(1) - \tilde\curve(1)}^{2}.
	\end{aligned}
\end{equation}
For a given time step $\tau \in (0, 1)$ we also define $\F \colon \AC^2 \to \mathbb{R}$ as
\begin{equation} \label{Louisa}
	\F(\curve, \tilde\curve) := \E(\curve) + \frac{1}{\tau} \D(\curve, \tilde\curve).
\end{equation}

We are now able to describe our minimizing movement scheme. For a given $\curve_0 \in \AC$ we define $(\dcurve_n)_n \subset \AC$ recursively as
\begin{equation} \label{Hattie}
	\left\{
	\begin{split}
		\dcurve_{n+1} &\in \argmin_{\curve \in \AC_{\dcurve_n}} \left\{ \E(\curve) + \frac{1}{\tau} \D(\curve, \dcurve_n) \right\}, \\
		\dcurve_0 &= \curve_0.
	\end{split}
	\right.
\end{equation}
Whenever $\tau > 0$ is fixed we shortly write $(\curve_n^\tau)$ as $(\curve_n)$. We are now going to apply the direct methods of the Calculus of Variations in order to show the well-definedness of the scheme in \eqref{Hattie}.

\begin{theorem}[Existence of step-by-step minimizers] \label{Gordon}
	For every $n \in \NN$ the problem in \eqref{Hattie} attains a minimum. Furthermore the following a priori bounds hold true for the sequence $(\dcurve_n)_n$:
	\begin{align}
		c \leq \abs{\dcurve_n(1) - \dcurve_n(0)} &\leq \len_{\dcurve_n} \leq \C \label{Bessie}, \\
		\int_{0}^{1} \curva_{\dcurve_n}^{2} \ds &\leq \C(\varepsilon). \label{Grace}
	\end{align}
\end{theorem}

\begin{proof}[Proof of Theorem \ref{Gordon}]
	Suppose that we have already proved the existence of $\curve_0$, ..., $\curve_n$ for some $n \in \NN$. By comparison and the definition of $\F$ we have
	\begin{equation} \label{Scott}
		\inf_{\curve \in \AC_{\curve_n}} \F(\curve, \curve_n) \leq \F(\curve_n, \curve_n) = \E(\curve_n).
	\end{equation}
	Furthermore, in the case $n > 1$, we iteratively derive again by comparison and the nonnegativity of $\D$
	\begin{equation} \label{Randy}
		\E(\curve_n) \leq \F(\curve_n, \curve_{n-1}) = \inf_{\curve \in \AC_{\curve_{n-1}}} \F(\curve, \curve_{n-1}) \leq \F(\curve_{n-1}, \curve_{n-1}) = \E(\curve_{n-1}) \leq \dots \leq \E(\curve_0).
	\end{equation}
	Furthermore, by the basic estimate
	\begin{equation*}
		-\log \abs{\curve(1) - \curve(0)} \geq - \abs{\curve(1) - \curve(0)} \geq -\len_\gamma,
	\end{equation*}
	and the very definition of $\E$ in \eqref{Franklin} we have that $\E$ is nonnegative on $\AC$. This fact combined with \eqref{Scott} and \eqref{Randy} then leads to
	\begin{equation*}
		0 \leq \inf_{\curve \in \AC_{\curve_n}} \F(\curve, \curve_n) \leq \E(\curve_0).
	\end{equation*}
	Consequently we can find a minimizing sequence $(\omu_{i}) \subset \AC_{\curve_n}$ such that
	\begin{gather}
		\lim_{i \to \infty} \F(\omu_{i}, \curve_n) = \inf_{\curve \in \AC_{\curve_n}} \F(\curve, \curve_n), \label{Nannie} \\
		\F(\omu_{i}, \curve_n) \leq \E(\curve_0) + 1 < \infty \text{ for all } i \in \NN. \label{Dylan}
	\end{gather}
	Our main aim now is to show that $\sup_{i} \norm{\omu_{i}}_{H^{2}} < \infty$. For this note that by \eqref{Dylan} and $\log t \leq \frac{t}{2}$ we first derive
	\begin{equation} \label{Bess}
		\begin{aligned}
			\E(\curve_0) + 1 \geq \F(\omu_{i}, \curve_n) \geq \E(\omu_{i}) &\geq -\frac{1}{2} \abs{\omu_{i}(1) - \omu_{i}(0)} + \len_{\omu_{i}} + \frac{\varepsilon \len_{\omu_{i}}}{2} \int_{0}^{1} \curva_{\omu_{i}}^{2} \ds \\
			&\geq \frac{\len_{\omu_{i}}}{2} + \frac{\varepsilon \len_{\omu_{i}}}{2} \int_{0}^{1} \curva_{\omu_{i}}^{2} \ds
		\end{aligned}
	\end{equation}
	Moreover, by the definition of $\D$, the non-negativity of $\E$ and \eqref{Dylan} we also get that
	\begin{equation} \label{Jordan}
		\frac{1}{2 \tau} \abs{\omu_{i}(0) - \curve_n(0)}^{2} \leq \D(\omu_{i}, \curve_n) \leq \F(\omu_{i}, \curve_n) \leq \E(\curve_0) + 1.
	\end{equation}
	Hence by the fundamental Theorem of Calculus (FTC ), the fact that $\abs{(\omu_i)_s} = \len_{\omu_i}$, $\tau \leq 1$, \eqref{Bess} and \eqref{Jordan} we derive
	\begin{equation} \label{Tony}
		\begin{aligned}
			\int_{0}^{1} \abs{\omu_{i}}^{2} \ds + \int_{0}^{1} \abs{(\omu_{i})_{s}}^{2} \ds &\leq ( \abs{\omu_{i}(0)} + \len_{\omu_{i}} )^{2} + \len_{\omu_{i}}^{2} \\
			&\leq (\abs{\curve_n(0)} + \abs{\omu_i(0) - \curve_n(0)} + \len_{\omu_i})^2 + \len_{\omu_i}^2 \\
			&\leq 2 \abs{\curve_n(0)}^{2} + \frac{4\tau}{2\tau} \abs{\omu_{i}(0) - \curve_n(0)}^{2} + 12 \left(\frac{ \len_{\omu_{i}}}{2}\right)^{2} \\
			&\leq 2 \abs{\curve_n(0)}^{2} + 4 (\E(\curve_0) + 1) + 12 (\E(\curve_0) + 1)^{2}.
		\end{aligned}
	\end{equation}
	Furthermore with \eqref{Micheal} applied to $\omu_{i}$ and \eqref{Bess} we follow
	\begin{equation} \label{Barbara}
		\int_{0}^{1} \abs{(\omu_{i})_{ss}}^{2} \ds = \int_{0}^{1} \abs{\len_{\omu_i} \curva_{\omu_i} \Rot((\omu_i)_s)}^2 \ds = \frac{16}{\varepsilon} \left(\frac{1}{2} \len_{\omu_{i}}\right)^{3} \left(\frac{\varepsilon}{2} \len_{\omu_{i}} \int_{0}^{1} \curva_{\omu_{i}}^{2} \ds\right) \leq \frac{16}{\varepsilon} ( \E(\curve_0) + 1 )^4.
	\end{equation}
	Combining \eqref{Tony} and \eqref{Barbara} eventually leads to $\sup_{i} \norm{\omu_{i}}_{H^{2}} < \infty$ as desired. By the weak compactness in $H^{2}$ and by the Sobolev embedding Theorem we can find $\omu \in H^{2}$ such that up to taking a subsequence
	\begin{align}
		\omu_{i} &\rightharpoonup \omu \text{ weakly in } H^{2} \label{Glen}, \\
		\omu_{i} &\to \omu \text{ in } W^{1, \infty}. \label{Lola}
	\end{align}
	We now wish to show that $\omu$ is admissible, which means $\omu \in \AC_{\curve_n}$.
	By \eqref{Lola} and $(\omu_{i}) \subset \AC_{\curve_n}$ we derive that $\omu$ also satisfies
	\begin{equation*}
		\abs{\omu_s} \equiv \len_{\omu}, \quad \braket{\omu_{s}}{(\curve_n)_{s}} \geq 0.
	\end{equation*}
	To prove that $\omu \in \AC_{\curve_n}$ it is left to show that $\omu(0) \neq \omu(1)$. This follows from \eqref{Lola} and \eqref{Dylan} which imply
	\begin{equation*}
		-\log \abs{\omu(1) - \omu(0)} = - \lim_{i \to \infty} \log \abs{\omu_{i}(1) - \omu_{i}(0)} \leq \limsup_{i \to \infty} \F(\omu_{i}, \curve_n) \leq \E(\curve_0) + 1.
	\end{equation*}
	Hence, by the monotonicity of $\log$
	\begin{equation} \label{Brent}
		\len_{\omu} \geq \abs{\omu(1) - \omu(0)} \geq c,
	\end{equation}
	where $\con$ is a constant only depending on $\curve_0$.
	It remains to show that $\omu$ is the desired minimizer. To this end note that
	\begin{align*}
		\int_{\omu_{i}} \curva_{\omu_{i}}^{2} \der \mathcal{H}^1 &= \int_{0}^{1} \left(\frac{\braket{(\omu_{i})_{ss}}{\Rot((\omu_i)_s)}}{\abs{(\omu_{i})_{s}}^3}\right)^{2} \abs{(\omu_{i})_{s}} \ds \\
		&= \int_{0}^{1} \braket{(\omu_{i})_{ss}}{\len_{\omu_{i}}^{-\frac{5}{2}} \Rot((\omu_i)_s)}^{2} \ds.
	\end{align*}
	By \eqref{Glen} and \eqref{Lola} the following convergences hold
	\begin{align*}
		(\omu_{i})_{ss} &\rightharpoonup \omu \text{ weakly in } L^{2}, \\
		\len_{\omu_{i}}^{-\frac{5}{2}} \Rot((\omu_i)_s) &\to \len_{\omu}^{-\frac{5}{2}} \Rot(\omu_{s}) \text{ in } L^{2}.
	\end{align*}
	Hence
	\begin{equation*}
		\braket{(\omu_{i})_{ss}}{\len_{\omu_{i}}^{-\frac{5}{2}} \Rot((\omu_i)_s)} \rightharpoonup \braket{\omu_{ss}}{\len_{\omu}^{-\frac{5}{2}} \Rot(\omu_{s})} \text{ weakly in } L^{2},
	\end{equation*}
	and therefore
	\begin{equation} \label{Violet}
		\liminf_{i \to \infty} \frac{\varepsilon}{2} \int_{\omu_{i}} \curva_{\omu_{i}}^{2} \der \mathcal{H}^1 \geq \frac{\varepsilon}{2} \int_{\omu} \curva_{\omu}^{2} \der \mathcal{H}^1.
	\end{equation}
	Furthermore, by \eqref{Lola} we have
	\begin{equation} \label{Gary}
		\lim_{i \to \infty} \F(\omu_{i}, \curve_n) - \frac{\varepsilon}{2} \int_{\omu_{i}} \curva_{\omu_{i}}^{2} \der \mathcal{H}^1 = \F(\omu, \curve_n) - \frac{\varepsilon}{2} \int_{\omu} \curva_{\omu}^{2} \der \mathcal{H}^1.
	\end{equation}
	Taking \eqref{Violet} and \eqref{Gary} together and using \eqref{Nannie} we derive
	\begin{equation} \label{Teresa}
		\inf_{\curve \in \AC_{\curve_n}} \F(\curve, \curve_n) = \lim_{i \to \infty} \F(\omu_{i}, \curve_n) \geq \F(\omu, \curve_n).
	\end{equation}
	Consequently $\curve_{n+1} := \omu$ is a desired minimizer. Passing to the limit as $i \to \infty$ in \eqref{Bess} we derive the upper bounds in \eqref{Bessie} and \eqref{Grace}. The lower bound in \eqref{Bessie} instead follows from \eqref{Brent}.
\end{proof}

\subsection{Compactness}

In this subsection we will derive important compactness results for interpolations of the sequence of step-by-step minimizers $(\dcurve)_n$, described in the next Definition.

\begin{definition}[Interpolations in time] \label{Rachel}
	Let $(\curve_n^\tau)_n$ be the sequence of step-by-step minimizers. We first define the piecewise constant functions related to $(\curve_n^\tau)_n$, namely $\curve^{\tau} \colon \ui \times \mathbb{R}_{+} \to \mathbb{R}^{2}$, $L^{\tau} \colon \mathbb{R}_{+} \to \mathbb{R}$ and $\curva^{\tau} \colon \mathbb{R}_{+} \to \mathbb{R}$ as
	\begin{equation*}
		\curve^{\tau}(s, t) := \curve_{\up{t/\tau}}^\tau(s), \quad L^{\tau}(t) := \len_{\curve_{\up{t/\tau}}^\tau}, \quad \curva^{\tau}(s, t) := \curva_{\curve_{\up{t/\tau}}}(s).
	\end{equation*}
	We introduce the following notation for their translations in time $\tilde\curve^{\tau} \colon \ui \times \mathbb{R}_{+} \to \mathbb{R}^{2}$, $\tilde{L}^{\tau} \colon \mathbb{R}_{+} \to \mathbb{R}$ and $\tilde{\curva}^{\tau} \colon \mathbb{R}_{+} \to \mathbb{R}$:
	\begin{equation*}
		\tilde\curve^{\tau}(s, t) := \curve^{\tau}(s, t-\tau), \quad \tilde{L}^{\tau}(t) := L^{\tau}(t-\tau), \quad \tilde{\curva}^{\tau}(s, t) := \curva^{\tau}(s, t-\tau).
	\end{equation*}
	Furthermore we define $\hat{\curve}^{\tau} \colon \ui \times \mathbb{R}_{+} \to \mathbb{R}^{2}$ as the piecewise affine interpolation in time of $(\curve_n^\tau)$ given by
	\begin{equation*}
		\hat{\curve}^{\tau}(s, t) := (\up{t/\tau} - t/\tau) \curve_{\down{t/\tau}}^\tau(s) + (t/\tau - \down{t/\tau}) \curve_{\up{t/\tau}}^\tau(s).
	\end{equation*}
	Finally we also define the piecewise affine interpolation in time of $(\len_{\curve_n^\tau})$ defined as the function $\hat{L}^\tau \colon \RR_+ \to \RR$ given by
	\begin{equation} \label{Rufat}
		\hat{L}^{\tau}(t) := (\up{t/\tau} - t/\tau) \len_{\curve_{\down{t/\tau}}^\tau} + (t/\tau - \down{t/\tau}) \len_{\curve_{\up{t/\tau}}^\tau}
	\end{equation}
	and we set $\dvelo \colon \ui \times \RR_+ \to \RR^2$ to be
	\begin{equation*}
		\dvelo(s, t) := \hat{\curve}^\tau_t(s, t).
	\end{equation*}
\end{definition}

In the following Lemma we will derive an important coupling relation between the tangential and the orthogonal projection of the speed. It will be eventually used in the proof of Lemma \ref{Martin} in order to derive a bound on the time-discrete speed.

\begin{lemma} \label{Austin}
	For every $s \in \ui$ and $t \in \mathbb{R}_{+}$ it holds that
	\begin{equation} \label{Don}
		\braket{\velo^{\tau}}{\tilde\curve^{\tau}_{s} + \curve^{\tau}_{s}}_{s} = (\tilde{L}^{\tau} + L^{\tau}) \hat{L}^{\tau}_{t} + \braket{\velo^{\tau}}{\tilde{L}^{\tau} \tilde{\curva}^{\tau} \Rot(\tilde\dcurve_s)+ L^{\tau} \curva^{\tau} \Rot(\dcurve_s)}.
	\end{equation}
\end{lemma}

\begin{proof}[Proof of Lemma \ref{Austin}]
	The derivation of the coupling relation \eqref{Don} is the result of the following computation: Since $\curve_n$ belongs to $\AC$ we have $\dcurve_s(s, t) \equiv L^\tau(t)$ for all $s \in \ui$ and $t \in \RR^+$. Defining $\omu^\tau := \dcurve_s + \tilde\dcurve_s$ we derive for all $s \in \ui$ and $t \in \RR_+$
	\begin{equation} \label{Anthony}
		\begin{aligned}
			(\tilde{L}^{\tau} + L^{\tau}) \hat{L}^{\tau}_{t} &= \frac{1}{\tau} (\tilde{L}^{\tau} + L^{\tau}) (L^{\tau} - \tilde{L}^{\tau}) = \frac{1}{\tau} \left((L^{\tau})^2 - (\tilde{L}^{\tau})^2\right) \\
			&= \frac{1}{\tau} (\braket{\curve^{\tau}_{s}}{\curve^{\tau}_{s}} - \braket{\tilde\curve^{\tau}_{s}}{\tilde\curve^{\tau}_{s}}) = \braket{\velo^{\tau}_s}{\omu^{\tau}}.
		\end{aligned}
	\end{equation}
	Furthermore by the product rule and \eqref{Micheal} we have
	\begin{equation} \label{Ricardo}
		\begin{aligned}
			\braket{\velo^{\tau}}{\omu^{\tau}}_{s} &= \braket{\velo^{\tau}_{s}}{\omu^{\tau}} + \braket{\velo^{\tau}}{\tilde\curve^{\tau}_{ss} + \curve^{\tau}_{ss}} \\
			&= \braket{\velo^{\tau}_{s}}{\omu^{\tau}} + \braket{\velo^{\tau}}{\tilde{L}^{\tau} \tilde{\curva}^{\tau} \Rot(\tilde\dcurve_s)+ L^{\tau} \curva^{\tau} \Rot(\dcurve_s)}.
		\end{aligned}
	\end{equation}
	Combining \eqref{Anthony} and \eqref{Ricardo} readily leads to \eqref{Don}.
\end{proof}

\begin{lemma} \label{Martin}
	\begin{equation} \label{Georgia}
		\int_{0}^{\infty} \int_{0}^{1} \abs{\velo^{\tau}}^{2} \ds \dt + \int_{0}^{\infty} \abs{\dvelo(0, t)}^{2} + \abs{\velo^{\tau}(1, t)}^{2} \dt \leq \C(\varepsilon).
	\end{equation}
\end{lemma}

\begin{proof}[Proof of Lemma \ref{Martin}]
	In order to shorten notation we shortly write as in the previous Lemma
	\begin{equation*}
		\omu^{\tau} := \tilde\curve^{\tau}_s + \curve^{\tau}_{s}.
	\end{equation*}
	Note that by $\curve_n \in \AC_{\curve_{n-1}}$ for all $n$, \eqref{Phoebe} and \eqref{Bessie} we have
	\begin{equation} \label{Isaiah}
		\begin{aligned}
			\abs{\omu^{\tau}}^{2} &= (\tilde{L}^{\tau})^2 + 2 \braket{\tilde\curve^{\tau}_{s}}{\curve^{\tau}_{s}} + (L^{\tau})^2 \\
			&\geq \tilde{L}^{\tau 2} + L^{\tau 2} \geq c > 0.
		\end{aligned}
	\end{equation}
	Furthermore comparison (see \eqref{Hattie}) we have that
	\begin{equation*}
		\frac{1}{\tau} \D(\curve_{n+1}, \curve_n) \leq \E(\curve_n) - \E(\curve_{n+1}).
	\end{equation*}
	Summing the above expression over $n$ and using the non-negativity of $\E$ we get
	\begin{equation} \label{Stella}
		\frac{1}{\tau} \sum_{n = 0}^{\infty} \D(\curve_{n+1}, \curve_n) \leq \sum_{n = 0}^{\infty} (\E(\curve_n) - \E(\curve_{n+1})) \leq \E(\curve_0) - \liminf_{N \to \infty} \E(\curve_{N+1}) \leq \E(\curve_0).
	\end{equation}
	We then compute
	\begin{equation} \label{Bernard}
		\begin{aligned}
			\frac{1}{\tau} \sum_{n = 0}^{\infty} \D(\curve_{n+1}, \curve_n) &= \sum_{n = 0}^{\infty} \frac{\tau}{4 \len_{\curve_n}} \int_{0}^{1} \braket*{\frac{\curve_{n+1} - \curve_n}{\tau}}{\Rot((\curve_n)_{s})}^{2} \ds \\
			&\quad + \sum_{n=0}^\infty \frac{\tau}{4 \len_{\curve_{n+1}}} \int_{0}^{1} \braket*{\frac{\curve_{n+1} - \curve_n}{\tau}}{\Rot((\curve_{n+1})_{s})}^{2} \ds \\
			&\quad + \sum_{n=0}^\infty \frac{\tau}{2} \left(\frac{\abs{\curve_{n+1}(0) - \curve_n(0)}}{\tau}\right)^{2} + \sum_{n=0}^\infty \frac{\tau}{2} \left(\frac{\abs{\curve_{n+1}(1) - \curve_n(1)}}{\tau}\right)^{2} \\
			&= \int_{0}^{\infty} \left( \frac{1}{4 \tilde{L}^{\tau}} \int_{0}^{1} \braket{\velo^{\tau}}{\Rot((\tilde\curve_s^{\tau}))}^{2} \ds + \frac{1}{4 L^{\tau}} \int_{0}^{1} \braket{\velo^{\tau}}{\Rot((\curve^{\tau}_{s})}^{2} \ds \right) \dt \\
			&\quad + \frac{1}{2} \int_{0}^{\infty} \abs{\dvelo(0, t)}^{2} + \abs{\velo^{\tau}(1, t)}^{2} \dt.
		\end{aligned}
	\end{equation}
	Combining \eqref{Stella} \eqref{Bernard} and \eqref{Bessie} leads to
	\begin{equation} \label{Paul}
			\int_{0}^{\infty} \left( \int_{0}^{1} \braket{\velo^{\tau}}{\Rot(\tilde\curve_s^{\tau})}^{2} \ds + \int_{0}^{1} \braket{\velo^{\tau}}{\Rot(\curve^{\tau}_{s})}^{2} \ds \right) \dt + \int_{0}^{\infty} \abs{\dvelo(0, t)}^{2} + \abs{\velo^{\tau}(1, t)}^{2} \dt \leq \C .
	\end{equation}
	Since \eqref{Paul} and \eqref{Isaiah} imply
	\begin{equation} \label{Bill}
		\int_{0}^{\infty} \int_{0}^{1} \braket*{\velo^{\tau}}{\frac{\Rot(\omu^{\tau})}{\abs{\omu^{\tau}}}}^{2} \ds \dt \leq \frac{1}{c^2} \int_{0}^{\infty} \int_{0}^{1} \braket{\velo^{\tau}}{\Rot(\omu^{\tau})}^{2} \ds \dt \leq \C ,
	\end{equation}
	it follows that
	\begin{equation} \label{Dustin}
		\int_{0}^{\infty} \int_{0}^{1} \braket*{\velo^{\tau}}{\frac{\Rot(\omu^\tau)}{\abs{\omu^{\tau}}}}^{2} \ds \dt + \int_{0}^{\infty} \abs{\dvelo(0, t)}^{2} + \int_{0}^{\infty} \abs{\velo^{\tau}(1, t)}^{2} \dt \leq \C .
	\end{equation}
	In order to obtain \eqref{Georgia} we are left to control
	\begin{equation*}
		\int_{0}^{\infty} \int_{0}^{1} \braket*{\velo^{\tau}}{\frac{\omu^{\tau}}{\abs{\omu^{\tau}}}}^{2} \ds \dt
	\end{equation*}
	from above. This will be achieved by employing \eqref{Don}. To this end we integrate \eqref{Don} in the curve parameter over $\ui$ and solve for $\hat{L}^{\tau}_{t}$ to derive
	\begin{equation*}
		\hat{L}^{\tau}_{t} = \frac{1}{\tilde{L}^{\tau} + L^{\tau}} \left( \braket{\velo^{\tau}}{\omu^{\tau}}|_{s = 0}^{1} - \int_{0}^{1} \braket{\velo^{\tau}}{\tilde{L}^{\tau} \tilde{\curva}^{\tau} \Rot(\tilde\dcurve_s)+ L^{\tau} \curva^{\tau} \Rot(\dcurve_s)} \ds \right)
	\end{equation*}
	Squaring both sides of the equality above, integrating them over $t \in \mathbb{R}_{+}$, and using \eqref{Bessie}, \eqref{Paul} and Hölder's inequality we get
	\begin{equation} \label{Darrell}
		\begin{aligned}
			\int_{0}^{\infty} (\hat{L}^{\tau}_{t})^2 \dt &\leq \C \int_{0}^{\infty} \abs{\velo^{\tau}(0, t)}^{2} + \abs{\velo^{\tau}(1, t)}^{2} \dt \\
			&\quad + \C \int_{0}^{\infty} \left(\int_{0}^{1} \braket{\velo^{\tau}}{\tilde{L}^{\tau} \tilde{\curva}^{\tau} \Rot(\tilde\dcurve_s)+ L^{\tau} \curva^{\tau} \Rot(\dcurve_s)} \ds\right)^{2} \dt \\
			&\leq \C + \C \int_{0}^{\infty} \left(\int_{0}^{1} (\tilde{\curva}^{\tau})^2 \ds\right) \left(\int_{0}^{1} \braket{\velo^{\tau}}{\Rot(\tilde\curve^{\tau}_{s})}^{2} \ds\right) \dt \\
			&\quad + \C \int_{0}^{\infty} \left(\int_{0}^{1} (\curva^{\tau})^2 \ds\right) \left(\int_{0}^{1} \braket{\velo^{\tau}}{\Rot(\dcurve_s)}^{2} \ds\right) \dt \\
			&\leq \C(\varepsilon) \left(1 + \int_{0}^{\infty} \int_{0}^{1} \braket{\velo^{\tau}}{\Rot(\tilde\curve^{\tau}_{s})}^{2} + \braket{\velo^{\tau}}{\Rot(\dcurve_s)}^{2} \ds \dt\right) \leq \C(\varepsilon),
		\end{aligned}
	\end{equation}
	where in the last inequality we used the $L^2$ bound of the curvature in \eqref{Grace}.
	Next we integrate \eqref{Don} again in the curve parameter but now over $[0, s]$:
	\begin{equation*}
		\braket{\velo^{\tau}}{\omu^{\tau}}(s, t) = \braket{\velo^{\tau}}{\omu^{\tau}}(0, t) + (\tilde{L}^{\tau} + L^{\tau}) \hat{L}^{\tau}_{t} s + \int_{0}^{s} \braket{\velo^{\tau}}{\tilde{L}^{\tau} \tilde{\curva}^{\tau} \Rot(\tilde\dcurve_s)+ L^{\tau} \curva^{\tau} \Rot(\dcurve_s)} \dtildes.
	\end{equation*}
	We then square both sides of the equality above, integrate over $(s, t)$ in $\ui \times \mathbb{R}_{+}$, as well as employ \eqref{Bessie}, \eqref{Grace}, \eqref{Paul}, \eqref{Darrell} and Hölder's Inequality in order to derive
	\begin{equation*}
		\begin{aligned}
			\int_{0}^{\infty} \int_{0}^{1} \braket{\velo^{\tau}}{\omu^{\tau}}^{2} \ds \dt &\leq \C \int_{0}^{\infty} \abs{\velo^{\tau}(0, t)}^{2} + \hat{L}^{\tau 2}_{t} \dt \\
			&\quad + \C \int_{0}^{\infty} \int_{0}^{1} \left(\int_{0}^{s} (\tilde{\curva}^{\tau})^2 (\tilde{s}, t) \dtildes\right) \left(\int_{0}^{s} \braket{\velo^{\tau}}{\Rot(\tilde\curve^{\tau}_s)}^{2}(\tilde{s}, t) \dtildes\right) \ds \dt \\
			&\quad + \C \int_{0}^{\infty} \int_{0}^{1} \left(\int_{0}^{s} (\curva^{\tau})^2 (\tilde{s}, t) \dtildes\right) \left(\int_{0}^{s} \braket{\velo^{\tau}}{\Rot(\dcurve_s)}^{2}(\tilde{s}, t) \dtildes\right) \ds \dt \\
			&\leq \C(\varepsilon) \left(1 + \int_{0}^{\infty} \int_{0}^{1} \braket{\velo^{\tau}}{\Rot(\tilde\curve^{\tau}_{s})}^{2} + \braket{\velo^{\tau}}{\Rot(\dcurve_s)}^{2} \ds \dt\right) \leq \C(\varepsilon).
		\end{aligned}
	\end{equation*}
	Hence by \eqref{Isaiah} we derive
	\begin{equation} \label{Minerva}
		\int_{0}^{\infty} \int_{0}^{1} \braket*{\velo^{\tau}}{\frac{\omu^{\tau}}{\abs{\omu^{\tau}}}}^{2} \ds \dt \leq \C \int_{0}^{\infty} \int_{0}^{1} \braket{\velo^{\tau}}{\omu^{\tau}}^{2} \ds \dt \leq \C(\varepsilon).
	\end{equation}
	In conclusion \eqref{Bill} and \eqref{Minerva} lead to \eqref{Georgia}.
\end{proof}

We continue by showing uniform Hölder continuity for the sequence of piecewise affine interpolations.

\begin{lemma} \label{George}
	For $0 \leq t_{1} < t_{2} < \infty$ it holds that
	\begin{equation} \label{Erik}
		\norm{\hat{\curve}^{\tau}(\cdot, t_{2}) - \hat{\curve}^{\tau}(\cdot, t_{1})}_{L^{2}} \leq \C(\varepsilon) (t_{2} - t_{1})^{\frac{1}{2}},
	\end{equation}
	and
	\begin{equation} \label{Aiden}
		\begin{aligned}
			\abs{\hat\curve^{\tau}(0, t_{2}) - \hat{\curve}^{\tau}(0, t_{1})} &\leq \C (t_{2} - t_{1})^{\frac{1}{2}}, \\
			\abs{\hat\curve^{\tau}(1, t_{2}) - \hat{\curve}^{\tau}(1, t_{1})} &\leq \C (t_{2} - t_{1})^{\frac{1}{2}}. \\
		\end{aligned}
	\end{equation}
	Furthermore for any $T > 0$ it holds
	\begin{equation} \label{Jeffrey}
		\norm{\curve^{\tau}}_{L^{\infty}(0, T; H^{2})} \leq \C(\varepsilon, T).
	\end{equation}
\end{lemma}

\begin{proof}[Proof of Lemma \ref{George}]
	By the absolute continuity of $\hat{\curve}^{\tau}(s, \cdot)$ for every $s \in \ui$, \eqref{Georgia}, Hölder's Inequality and Fubini's Theorem we derive for all $0 \leq t_{1} < t_{2} < \infty$
	\begin{equation*}
		\begin{aligned}
			\norm{\hat{\curve}^{\tau}(\cdot, t_{2}) - \hat{\curve}^{\tau}(\cdot, t_{1})}_{L^{2}} &= \left(\int_{0}^{1} \abs*{ \int_{t_{1}}^{t_{2}} \velo^{\tau} \dt}^{2} \ds\right)^{\frac{1}{2}} \\
			&\leq \left(\int_{0}^{1} (t_{2} - t_{1}) \int_{t_{1}}^{t_{2}} \abs{\velo^{\tau}}^{2} \dt \ds\right)^{\frac{1}{2}} \\
			&\leq \left(\int_{0}^{\infty} \norm{\velo^{\tau}}_{L^{2}}^{2} \dt\right)^{\frac{1}{2}} (t_{2} - t_{1})^{\frac{1}{2}} \leq \C(\varepsilon) (t_{2} - t_{1})^{\frac{1}{2}},
		\end{aligned}
	\end{equation*}
	hence \eqref{Erik} follows. The proof of \eqref{Aiden} follows similarly. Let us now fix $T > 0$, then with the definition of $\hat{\curve}^{\tau}$ and \eqref{Erik} we can derive for any $0 \leq t \leq T$
	\begin{equation} \label{Emma}
		\norm{\hat{\curve}^{\tau}(\cdot, t)}_{L^{2}} \leq \norm{\hat{\curve}^{\tau}(\cdot, t) - \hat{\curve}^{\tau}(\cdot, 0) }_{L^{2}} + \norm{\curve_0}_{L^{2}} \leq \C(\varepsilon) T^{\frac{1}{2}} + \C \leq \C(\varepsilon, T).
	\end{equation}
	Applying \eqref{Emma} for $t = \tau n$, $n \in \NN$, and using the definition of $\hat{\curve}^{\tau}$ we see that
	\begin{equation} \label{Rebecca}
		\norm{\curve^{\tau}(\cdot, t)}_{L^{2}(0, T; L^{2})} \leq \C(\varepsilon, T).
	\end{equation}
	Furthermore by \eqref{Bessie} and \eqref{Grace} we have
	\begin{equation} \label{Nina}
		\norm{\curve^{\tau}_{s}}_{L^{\infty}(\mathbb{R}_{+}; H^{1})} \leq \C ,
	\end{equation}
	which leads to \eqref{Jeffrey}.
\end{proof}

Throughout the paper we will employ the following interpolation inequality. For a proof we refer to \cite{adams2003sobolev} and Theorem 6.4 in \cite{fonseca2012motion}.

\begin{theorem}[Interpolation Inequality] \label{Jennie}
	Let $\Omega \subset \mathbb{R}^{n}$ be a bounded open set satisfying the cone condition. Let $i$, $j$, and $m$ be integers such that $0 \leq i \leq j \leq m$. Let $1 \leq p \leq q < \infty$ if $(m - j)p \geq n$, or let $1 \leq p \leq q \leq \infty$ if $(m - j)p > n$. Then, there exists $\C > 0$ such that for all $u \in W^{m, p}(\Omega)$ it holds
	\begin{equation} \label{Millie}
		\norm{D^{j}u}_{L^{q}(\Omega)} \leq \C \left(\norm{D^{m}u}_{L^{p}(\Omega)}^{\theta} \norm{D^{i}u}_{L^{p}(\Omega)}^{1 - \theta} + \norm{D^{i}u}_{L^{p}(\Omega)}\right),
	\end{equation}
	where
	\begin{equation*}
		\theta := \frac{1}{m - i} \left(\frac{n}{p} - \frac{n}{q} + j - i\right).
	\end{equation*}
\end{theorem}

Thanks to the uniform $L^2$-bound on the curvature in \eqref{Grace} we will improve the Hölder continuity results from the previous Lemma by interpolation.

\begin{lemma} \label{Ivan}
	For any, $\alpha \in (0, \frac{1}{2})$, $T > 0$ and $0 \leq t_{1} < t_{2} \leq T$ it holds that
	\begin{equation} \label{Leon}
		\norm{\hat{\curve}^{\tau}(\cdot, t_{2}) - \hat{\curve}^{\tau}(\cdot, t_{1})}_{\C ^{1, \alpha}} \leq \C(\varepsilon, T) (t_{2} - t_{1})^{\frac{1 - 2\alpha}{8}}.
	\end{equation}
\end{lemma}

\begin{remark}[to Lemma \ref{Ivan}] \label{Johanna}
	Take any $\alpha \in (0, \frac{1}{2})$, and $T > 0$. Using \eqref{Bessie} and \eqref{Leona} we derive
	\begin{equation*}
		\begin{aligned}
			\abs{\braket{\curve^{\tau}_{s}}{\tilde\curve^{\tau}_{s}} - L^{\tau} \tilde{L}^{\tau}} &= \abs{\braket{\curve^{\tau}_{s} - \tilde\curve^{\tau}_{s}}{\tilde\curve^{\tau}_{s}} + (\tilde{L}^{\tau})^2 - L^{\tau} \tilde{L}^{\tau}} \\
			&\leq \tilde{L}^{\tau} \left(\norm{\curve^{\tau}_{s} - \tilde\curve^{\tau}_{s}}_{L^{\infty}} + \int_{0}^{1} \abs{\curve^{\tau}_{s} - \tilde\curve^{\tau}_{s}} \ds\right) \\
			&\leq \C \norm{\curve^{\tau}_{s} - \tilde\curve^{\tau}_{s}}_{L^{\infty}} \leq \C(\varepsilon, T) \tau^{\frac{1 - 2\alpha}{8}} \overset{\tau \to 0}{\to} 0.
		\end{aligned}
	\end{equation*}
	Hence by \eqref{Bessie} there exists $\tau_{0} := \tau(\varepsilon, T) > 0$ such that for all $\tau < \tau_{0}$ we have
	\begin{equation*}
		\braket{\curve^{\tau}_{s}}{\tilde\curve^{\tau}_{s}} > 0.
	\end{equation*}
	In particular we derive the following crucial result: For $\tau < \tau_{0}$ and $n \leq \lfloor T/\tau \rfloor$ the step-by-step minimizer $\curve^{\tau}_{n+1}$ satisfies
	\begin{equation} \label{Cole}
		\curve^{\tau}_{n+1} \in \argmin_{\curve \in \AC} \F(\curve, \curve^{\tau}_{n}).
	\end{equation}
	This will become relevant once we compute the Euler-Lagrange equation corresponding to the step-by-step minimization \eqref{Hattie} as \eqref{Cole} tells us that the additional angle constraint in \eqref{Hattie} is not influencing the minimization, at least for $\tau < \tau_{0}$ and $n \leq \lfloor T/\tau \rfloor$.
\end{remark}

\begin{proof}[Proof of Lemma \ref{Ivan}]
	Fix $\alpha \in (0, \frac{1}{2})$, $T > 0$ and $0 \leq t_{1} < t_{2} \leq T$. In order to shorten notation, we define
	\begin{equation*}
		\ddcurve := \hat{\curve}^{\tau}(\cdot, t_{2}) - \hat{\curve}^{\tau}(\cdot, t_{1}).
	\end{equation*}
	Using the interpolation inequality \eqref{Millie} for $\ddcurve$ with $n = 1$, $i = 0$, $j = 1$, $m = 2$, $p = 2$ and $q = \infty$ we derive
	\begin{equation} \label{Katharine}
		\norm{\ddcurve_{s}}_{L^\infty} \leq \C \left(\norm{\ddcurve_{ss}}_{L^{2}}^{\frac{3}{4}} \norm{\ddcurve}_{L^{2}}^{\frac{1}{4}} + \norm{\ddcurve}_{L^{2}}\right).
	\end{equation}
	By the very definition of $\ddcurve$, thanks to \eqref{Erik} and \eqref{Jeffrey} we can control the right-hand side of \eqref{Katharine} as follows
	\begin{equation} \label{Catherine}
		\begin{aligned}
			\norm{\ddcurve_{s}}_{L^{\infty}} &\leq \C(\varepsilon, T) \left((t_{2} - t_{1})^{\frac{1}{8}} + (t_{2} - t_{1})^{\frac{1}{2}}\right) \\
			&= \C(\varepsilon, T) \left(1 + (t_{2} - t_{1})^{\frac{1}{2} - \frac{1}{8}}\right) (t_{2} - t_{1})^{\frac{1}{8}} \\
			&\leq \C(\varepsilon, T) (t_{2} - t_{1})^{\frac{1}{8}}.
		\end{aligned}
	\end{equation}
	Note that in the last inequality we have used the fact that $t_{2}$, $t_{2}$ are in the bounded interval $[0, T]$. By the fundamental theorem of calculus, \eqref{Aiden} and \eqref{Catherine} we also derive
	\begin{equation} \label{Lawrence}
		\begin{aligned}
			\norm{\ddcurve}_{L^{\infty}} &\leq \abs{\ddcurve(0)} + \int_{0}^{1} \abs{\ddcurve_{s}} \ds \\
			&\leq \C (t_{2} - t_{1})^{\frac{1}{2}} + \norm{\ddcurve_{s}}_{L^{\infty}} \\
			&\leq \C (t_{2} - t_{1})^{\frac{1}{2}} + \C(\varepsilon, T) (t_{2} - t_{1})^{\frac{1}{8}} \\
			&\leq \C(\varepsilon, T) (t_{2} - t_{1})^{\frac{1}{8}}.
		\end{aligned}
	\end{equation}
	In order to conclude it remains to control the Hölder semi-norm $\abs{\ddcurve_{s}}_{\alpha}$. By Morrey's Inequality, \eqref{Catherine} and \eqref{Jeffrey} we have that
	\begin{equation} \label{Sallie}
		\begin{aligned}
			\abs{\ddcurve_{s}}_{\alpha} &= \sup_{s_{1}, s_{2} \in I} \frac{\abs{\ddcurve_{s}(s_{2}) - \ddcurve_{s}(s_{1})}}{\abs{s_{2} - s_{1}}^{\alpha}} \\
			&= \left(\sup_{s_{1}, s_{2} \in I} \frac{\abs{\ddcurve_{s}(s_{2}) - \ddcurve_{s}(s_{1})}}{\abs{s_{2} - s_{1}}^{\frac{1}{2}}}\right)^{2 \alpha} \sup_{s_1, s_2 \in I} \abs{\ddcurve_{s}(s_{2}) - \ddcurve_{s}(s_{1})}^{1 - 2\alpha} \\
			&\leq \C \abs{\ddcurve_{s}}_{\frac{1}{2}}^{2 \alpha} \norm{\ddcurve_{s}}_{L^{\infty}}^{1 - 2\alpha} \\
			&\leq \C \norm{\ddcurve}_{H^{2}}^{2 \alpha} \norm{\ddcurve_{s}}_{L^{\infty}}^{1 - 2 \alpha} \leq \C(\varepsilon, T) (t_{2} - t_{1})^{\frac{1 - 2\alpha}{8}}
		\end{aligned}
	\end{equation}
	Combining \eqref{Lawrence}, \eqref{Catherine} and \eqref{Sallie} results in \eqref{Leon}.
\end{proof}

We combine the previous Lemmas of this subsection to derive the initial compactness result.

\begin{theorem}[Initial Compactness] \label{Emilie}
	Given $\hat{\curve}^\tau$ and $\dcurve$ as in Definition \ref{Rachel}, there exists $\curve \colon \ui \times \mathbb{R}_{+} \to \mathbb{R}^{2}$ such that for any $\alpha \in (0, \frac{1}{2})$ and $\beta \in (0, \frac{1 - 2\alpha}{8})$, up to subsequences it holds
	\begin{align}
		\hat{\curve}^{\tau} &\to \curve \text{ in } \C ^{0, \beta}_{\loc}(\mathbb{R}_{+}; \C ^{1, \alpha}), \label{Leona} \\
		\curve^{\tau} &\to \curve \text{ in } L^{\infty}_{\loc}(\mathbb{R}_{+}; \C ^{1, \alpha}), \label{Jeremiah}
	\end{align}
	and
	\begin{align}
		\hat{\curve}^{\tau} &\rightharpoonup \curve \text{ weakly in } H^{1}_{\loc}(\mathbb{R}_{+}; L^{2}), \label{Cody} \\
		\begin{split} \label{Cory}
			\hat{\curve}^{\tau}(0, \cdot) &\rightharpoonup \curve(0, \cdot) \text{ weakly in } H^{1}_{\loc}(\mathbb{R}_{+}; \mathbb{R}^{2}). \\
			\hat{\curve}^{\tau}(1, \cdot) &\rightharpoonup \curve(1, \cdot) \text{ weakly in } H^{1}_{\loc}(\mathbb{R}_{+}; \mathbb{R}^{2}).
		\end{split}
	\end{align}
\end{theorem}

\begin{proof}[Proof of Theorem \ref{Emilie}]
	The proof of \eqref{Leona} follows from \eqref{Leon} and a standard diagonal sequence argument. For this let $(T_{k}) \subset \mathbb{R}_{+}$ be an auxiliary sequence with $T_{k} \uparrow \infty$. By \eqref{Leon} and the Arzelá-Ascoli Theorem there exists $(\tau^{(0)}_{n})$ converging to $0$ and $\curve^{(0)} \colon \ui \times [0, T_{0}] \to \mathbb{R}^{2}$ such that for any $\alpha \in (0, \frac{1}{2})$, $\beta \in (0, \frac{1 - 2\alpha}{8})$ as $n \to \infty$ we have
	\begin{equation} \label{Mason}
		\hat{\curve}^{\tau^{(0)}_{n}} \to \curve^{(0)} \text{ in } \C ^{0, \beta}(0, T_{0}; \C ^{1, \beta}).
	\end{equation}
	Now for every $k \in \NN$ we apply the Arzelá-Ascoli Theorem to the sequence $\hat{\curve}^{\tau^{(k)}_n}$ to construct $(\tau_n^{(k + 1)})_n$ as a subsequence of $(\tau_n^{(k)})_n$ and $\curve^{(k + 1)} \colon \ui \times [0, T_{k+1}] \to \mathbb{R}^{2}$ such that for any $\alpha \in (0, \frac{1}{2})$, $\beta \in (0, \frac{1 - 2\alpha}{8})$ we have as $n \to \infty$
	\begin{equation} \label{Harvey}
		\hat{\curve}^{\tau^{(k+1)}_{n}} \to \curve^{(k+1)} \text{ in } \C ^{0, \beta}(0, T_{k+1}; \C ^{1, \alpha}).
	\end{equation}
	Note that, as $\C ^{0, \alpha}(0, T_{k+1}; \C ^{1, \alpha})$-convergence implies $\C ^{0, \alpha}(0, T_{k}; \C ^{1, \alpha})$-convergence for any $k \in \NN$, we have
	\begin{equation*}
		\curve^{(k+1)}\|_{[0, T_{k}]} = \curve^{(k)}
	\end{equation*}
	for any $k \in \NN$. Hence we can define $\curve \colon \ui \times \RR_+ \to \RR^2$ as
	\begin{equation*}
		\curve\|_{[0, T_k]} := \curve^{(k)} \text{ for } k \in \NN.
	\end{equation*}
	By \eqref{Mason} and \eqref{Harvey} we have along the diagonal sequence $\tau_{n} := \tau^{(n)}_{n}$ that for any $\alpha \in (0, \frac{1}{2})$, $\beta \in (0, \frac{1 - 2\alpha}{8})$
	\begin{equation*}
		\hat{\curve}^{\tau_{n}} \to \curve \text{ in } \C ^{0, \beta}_{\loc}(\mathbb{R}_{+}; \C ^{1, \alpha}).
	\end{equation*}
	From this point on we assume that we have already extracted the subsequence $(\hat{\curve}^{\tau_{n}})$ and we will denote it, for the sake of shorter notation, just by $(\hat{\curve}^{\tau})$. By the definition of $\hat{\curve}^{\tau}$ and $\dcurve$, and thanks to \eqref{Leon}, for any $\alpha \in (0, \frac{1}{2})$, $T > 0$ and $0 \leq t \leq T$
	\begin{equation} \label{Annie}
		\norm{\curve^{\tau}(\cdot, t) - \hat{\curve}^{\tau}(\cdot, t)}_{\C ^{1, \alpha}} \leq \C(\varepsilon, T) \tau^{\frac{1 - 2\alpha}{8}} \overset{\tau \to 0}{\to} 0.
	\end{equation}
	As a consequence of \eqref{Leona} and \eqref{Annie} we can deduce \eqref{Jeremiah}. Thanks to \eqref{Georgia} and the already proven convergence \eqref{Leona} we have up to a further subsequence that \eqref{Cody} and \eqref{Cory} hold true.
\end{proof}

We next wish to compute the first variation of the minimization problem
\begin{equation*}
	\min_{\curve \in \AC} \F(\curve, \tilde\curve),
\end{equation*}
for some fixed $\tilde\curve \in \AC$. Due to the non-linearity of the speed constraint of $\AC$ the additive variation $\curve + \delta \oeta$, with $\curve \in \AC$, $\delta > 0$ and $\oeta \in H^2(\ui, \RR^2)$, is in general not admissible. In the next Lemma \ref{Elsie} we will show that there exists a reparameterization $P \colon \ui \to \ui$ (depending on $\delta$) such that $(\curve + \delta \oeta) \circ P \in \AC$.

\begin{lemma}[Admissible variations in $\AC$] \label{Elsie}
	For $\curve \in \AC$, given $\oeta \in H^{2}(\ui; \mathbb{R}^{2})$ and $\delta$ such that $0 < \delta < \frac{\len_\gamma}{\norm{\oeta_{s}}_{L^{\infty}}}$ there exists a unique $P(\delta, \cdot) \colon \ui \to \ui$ such that $\omu(\delta, \cdot) \colon \ui \to \mathbb{R}^{2}$, defined as
	\begin{equation} \label{Maude}
		\omu(\delta, s) := (\curve + \delta \oeta)(P(\delta, s)),
	\end{equation}
	satisfies
	\begin{equation} \label{Dale}
		\omu(\delta, \cdot) \in \AC, \quad \omu(\delta, 0) = \curve(0) + \delta \oeta(0).
	\end{equation}
	Furthermore we have
	\begin{align}
		\LM_\delta(\curve, \oeta) &= \frac{1}{\len_\gamma^{2}} \left(s \int_{0}^{1} \braket{\curve_s}{\oeta_{s}} \dtildes - \int_{0}^{s} \braket{\curve_s}{\oeta_{s}} \dtildes\right), \label{Myrtle} \\
		\LM_{s}(0, s) &= 1, \label{Tommy} \\
		\LM_{s \delta}(\curve, \oeta) &= \frac{1}{\len_\gamma^{2}} \left(\int_{0}^{1} \braket{\curve_s}{\oeta_{s}} \dtildes - \braket{\curve_s(s)}{\oeta_{s}(s)}\right). \label{Elva}
	\end{align}
\end{lemma}

\begin{proof}[Proof of Lemma \ref{Elsie}]
	Let us consider the auxiliary differentiable function $\Fit(\delta, \cdot) \colon \ui \to \mathbb{R}$ given by
	\begin{equation} \label{Eldar}
		\Fit(\delta, s') := \frac{\int_{0}^{s'} \abs{\curve_s + \delta \oeta_{s}} \dtildes}{\len_{\curve + \delta \oeta}}.
	\end{equation}
	Then for $0 < \delta < \frac{\len_\gamma}{\norm{\oeta_{s}}_{L^{\infty}}}$ we have
	\begin{equation*}
		\abs{\curve_s + \delta \oeta_{s}} \geq \abs{\curve_s} - \delta \norm{\oeta_{s}}_{L^{\infty}} = \len_\gamma - \delta \norm{\oeta_{s}}_{L^{\infty}} > 0.
	\end{equation*}
	Hence it follows that $F_{s'} > 0$. Together with $\Fit(\delta, 0) = 0$ and $\Fit(\delta, 1) = 1$ this implies that $\Fit(\delta, \cdot)$ is a diffeomorphism from $\ui$ to $\ui$. Therefore we can consider $\LM(\delta, \cdot) \colon \ui \to \ui$ defined as
	\begin{equation*}
		\LM(\delta, s) := \Fit(\delta, \cdot)^{-1}(s).
	\end{equation*}
	We now check that for such a choice of $\LM$ the statement of the Lemma holds. Let us define $\omu(\delta, \cdot)$ as in \eqref{Maude}. From the definition of $\LM$ we derive $\LM(\delta, 0) = 0$. Hence
	\begin{equation} \label{Nathaniel}
		\omu(\delta, 0) = \curve(0) + \delta \oeta(0)
	\end{equation}
	From $\Fit(\delta, \LM(\delta, s)) = s$, by the chain rule we derive that
	\begin{align}
		F_{s'}(\delta, \LM(\delta, s)) P_{s}(\delta, s) &= 1 \label{Jimmy} \\
		F_{\delta}(\delta, \LM(\delta, s)) + F_{s'}(\delta, \LM(\delta, s)) P_{\delta}(\delta, s) &= 0. \label{Earl}
	\end{align}
	Moreover we also have that
	\begin{align}
		F_{s'}(\delta, s') &= \frac{\abs{\curve_s(s') + \delta \oeta_{s}(s')}}{\len_{\curve + \delta \oeta}}, \label{Alejandro} \\
		\begin{split}
			F_{\delta}(\delta, s') &= \frac{1}{\len_{\curve + \delta \oeta}} \int_{0}^{s'} \braket*{\frac{\curve_s + \delta \oeta_{s}}{\abs{\curve_s + \delta \oeta_{s}}}}{\oeta_{s}} \dtildes - \frac{1}{\len_{\curve + \delta \oeta}^{2}} \int_{0}^{s'} \abs{\curve_s + \delta \oeta_{s}} \dtildes \int_{0}^{1} \braket*{\frac{\curve_s + \delta \oeta_{s}}{\abs{\curve_s + \delta \oeta_{s}}}}{\oeta_{s}} \dtildes.
		\end{split} \label{Hettie}
	\end{align}
	Hence \eqref{Jimmy} and \eqref{Alejandro} imply
	\begin{equation} \label{Leah}
		P_{s}(\delta, s) = \frac{\len_{\curve + \delta \oeta}}{\abs{(\curve_s + \delta \oeta_{s})(\LM(\delta, s))}},
	\end{equation}
	by which \eqref{Tommy} follows. Furthermore, from the same equation we have
	\begin{equation*}
		\abs{(\omu(\delta, \cdot))_{s}(s)} = \abs{(\curve + \delta \oeta)(\LM(\delta, s))} \abs{P_{s}(\delta, s)} = \len_{\curve + \delta \oeta}
	\end{equation*}
	and therefore $\omu \in \AC$. It remains to check \eqref{Myrtle} and \eqref{Elva}. We use \eqref{Earl}, \eqref{Alejandro}, \eqref{Hettie} and $\LM(0, s) = s$ in order to compute
	\begin{equation*}
		P_\delta(\curve, \oeta) = -\frac{F_{\delta}(0, \LM(0, s))}{F_{s'}(0, \LM(0, s))} = -F_{\delta}(0, s) = \frac{1}{\len_\gamma^{2}} \left(s \int_{0}^{1} \braket{\curve_s}{\oeta_{s}} \dtildes - \int_{0}^{s} \braket{\curve_s}{\oeta_{s}} \dtildes\right)
	\end{equation*}
	that is \eqref{Myrtle}. In order to conclude the proof we differentiate \eqref{Leah} with respect to $\delta$ and we get
	\begin{align*}
		P_{s \delta}(\delta, s) &= \frac{1}{\abs{(\curve_s + \delta \oeta_{s})(\LM(\delta, s))}} \int_{0}^{1} \braket*{\frac{\curve_s + \delta \oeta_{s}}{\abs{\curve_s + \delta \oeta_{s}}}}{\oeta_{s}} \dtildes \\
		&\quad - \frac{\len_{\curve + \delta \oeta}}{\abs{(\curve_s + \delta \oeta_{s})(\LM(\delta, s))}^{3}} \langle (\curve_s + \delta \oeta_{s})(\LM(\delta, s)) \mid \curve_{ss}(\LM(\delta, s)) P_{\delta}(\delta, s) \\
		&\quad + \oeta_{s}(\LM(\delta, s)) + \delta \oeta_{ss}(\LM(\delta, s)) P_{\delta}(\delta, s) \rangle.
	\end{align*}
	Plugging in $\delta = 0$ above and $\braket{\curve_s}{\curve_{ss}} = 0$ eventually leads to \eqref{Elva}.
\end{proof}

\begin{remark}
	We wish to provide intuition behind formula \eqref{Eldar}. Suppose that there exists a $\LM(\delta, \cdot) \colon \ui \to \ui$ such that $\omu(\delta, \cdot)$ as defined in \eqref{Maude} satisfies \eqref{Dale}. Hence we can follow that
	\begin{equation*}
		\int_{0}^{\LM(\delta, s)} \abs{\curve_s + \delta \oeta_{s}} \dtildes = \int_0^s \abs{\omu_s(\delta, \cdot)} \dtildes = \len_{\curve + \delta \oeta} s
	\end{equation*}
	for all $s \in \ui$. After dividing by $\len_{\curve + \delta \oeta}$ above, we see that $\LM(\delta, \cdot)$ is the inverse of
	\begin{equation*}
		\Fit(\delta, s') := \frac{\int_{0}^{s'} \abs{\curve_s + \delta \oeta_{s}} \dtildes}{\len_{\curve + \delta \oeta}},
	\end{equation*}
	as long as one such inverse exists.
\end{remark}

\begin{definition} \label{Bruce}
	Given $\curve$, $\oeta \in H^1(\ui; \RR^2)$ we define $\LM_1(\curve, \oeta) \colon \ui \to \RR$ and $\LM_2(\curve, \oeta) \colon \ui \to \RR$ as
	\begin{align}
		\LM_1(\curve, \oeta)(s) &:= \frac{1}{\len_\gamma^2} \left(s \int_{0}^{1} \braket{\curve_s}{\oeta_{s}} \dtildes - \int_{0}^{s} \braket{\curve_s}{\oeta_{s}} \dtildes\right), \label{def:P1} \\
		\LM_2(\curve, \oeta)(s) &= \frac{1}{\len_\gamma^2} \left(\int_{0}^{1} \braket{\curve_s}{\oeta_{s}} \dtildes - \braket{\curve_s}{\oeta_{s}}\right). \label{def:P2}
	\end{align}
\end{definition}

We are finally ready to compute the first variation of the minimization problem \eqref{Cole} eventually leading to the weak formulation of the time-discrete evolution in Theorem \ref{Betty}.

\begin{lemma}[First variation] \label{Jerry}
	Fix $\tilde\curve \in \AC$ and let
	\begin{equation} \label{Nora}
		\curve \in \argmin_{\omu \in \AC} \F(\omu, \tilde\curve).
	\end{equation}
	Then for all $\oeta \in \C ^\infty(\ui; \mathbb{R}^{2})$ it holds that
	\begin{equation} \label{Alberta}
		\E(\curve, \oeta) + \D(\curve, \oeta) + \Err(\curve, \oeta) = 0,
	\end{equation}
	where
	\begin{align}
		\E(\curve, \oeta) &:= \int_{0}^{1} \frac{\varepsilon}{\len_\gamma^{3}} \braket{\curve_{ss}}{\oeta_{ss}} + \frac{1}{\len_\gamma} (1 - \frac{3 \varepsilon}{2} \curva_\gamma^{2}) \braket{\curve_s}{\oeta_{s}} \ds - \braket*{\frac{\curve(1) - \curve(0)}{\abs{\curve(1) - \curve(0)}^2}}{\oeta(1) - \oeta(0)}, \label{Lester} \\
		\begin{split}
			\D(\curve, \oeta) &:= \frac{1}{2 \len_{\tilde\curve}} \int_{0}^{1} \braket*{\frac{\curve - \tilde\curve}{\tau}}{\Rot(\tilde\curve_s)} \braket{\Rot(\tilde\curve_s)}{\oeta} \ds + \frac{1}{2 \len_\gamma} \int_{0}^{1} \braket*{\frac{\curve - \tilde\curve}{\tau}}{\Rot(\curve_s)} \braket{\Rot(\curve_s)}{\oeta} \ds \\
			&\quad + \braket*{\frac{\curve(0) - \tilde\curve(0)}{\tau}}{\oeta(0)} + \braket*{\frac{\curve(1) - \tilde\curve(1)}{\tau}}{\oeta(1)},
		\end{split} \label{Christopher} \\
		\begin{split}
			\Err(\curve, \oeta) &:= \frac{1}{2 \len_{\tilde\curve}} \int_{0}^{1} \braket*{\frac{\curve - \tilde\curve}{\tau}}{\Rot(\tilde\curve_s)} \braket{\Rot(\tilde\curve_s)}{\LM_1(\curve, \oeta) \curve_s} \ds \\
			&\quad + \frac{1}{2 \len_{\tilde\curve}} \int_{0}^{1} \braket*{\frac{\curve - \tilde\curve}{\tau}}{\Rot(\tilde\curve_s)} \braket{\curve - \tilde \curve}{\LM_1(\curve, \oeta) \Rot(\curve_{ss}) + \LM_2(\curve, \oeta) \Rot(\curve_s) + \Rot(\oeta_s)} \ds \\
			&\phantom{:=} - \frac{1}{4 \len_\gamma^{3}} \int_{0}^{1} \braket{\curve_s}{\oeta_{s}} \ds \int_{0}^{1} \braket*{\frac{\curve - \tilde\curve}{\tau}}{\Rot(\curve_s)} \braket{\curve - \tilde\curve}{\Rot(\curve_s)} \ds.
		\end{split} \label{Connor}
	\end{align}
\end{lemma}

\begin{proof}[Proof of Lemma \ref{Jerry}]
	By the minimality of $\curve$, $\frac{d}{d\delta}|_{\delta = 0}\F(\omu(\delta, \cdot)) = 0$ with $\omu(\delta, \cdot)$ as defined in Lemma \ref{Elsie}. It remains to show that
	\begin{equation*}
		\frac{d}{d\delta}|_{\delta = 0}\F(\omu(\delta, \cdot)) = E(\curve, \oeta) + D(\curve, \oeta) + Er\Rot(\curve, \oeta).
	\end{equation*}
	Given $\omu \in \AC$, for readers convenience, we split the dissipation $\D$ defined in \eqref{Cecilia} in the following three terms
	\begin{align*}
		\D_{1}(\omu) &:= \frac{1}{2 \tau} \abs{\omu_{0} - \tilde\curve(0)}^{2} + \frac{1}{2 \tau} \abs{\omu(1) - \tilde\curve(1)}^{2}, \\
		\D_{2}(\omu) &:= \frac{1}{4 \tau \len_{\tilde\curve}} \int_{0}^{1} \braket{\omu - \tilde\curve}{\Rot(\tilde\curve_s)}^{2} \ds, \\
		\D_{3}(\omu) &:= \frac{1}{4 \tau \len_{\omu}} \int_{0}^{1} \braket{\omu - \tilde\curve}{\Rot(\omu_s)}^{2} \ds.
	\end{align*}
	From the very definition \eqref{Louisa} of $\F$ we can then write
	\begin{equation*}
		\F(\omu, \tilde\curve) = \E(\omu) + \D_{1}(\omu) + \D_{2}(\omu) + \D_{3}(\omu).
	\end{equation*}

	\textit{First variation of $\E$}:
	\begin{equation*}
		\E(\omu(\delta, \cdot)) = -\log \abs{\curve(1) - \curve(0) + \delta (\oeta(1) - \oeta(0))} + \int_{0}^{1} \frac{\varepsilon}{2} \frac{\braket{\curve_{ss} + \delta \oeta_{ss}}{\Rot(\curve_s) + \delta \Rot(\oeta_s)}^2}{\abs{\curve_s + \delta \oeta_{s}}^{5}} + \abs{\curve_s + \delta \oeta_{s}} \ds.
	\end{equation*}
	By the dominated convergence Theorem and thanks to $\curve_{ss} = \len_\gamma \curva_\gamma \Rot(\curve_s)$ we derive
	\begin{align*}
		\frac{d}{d \delta}|_{\delta = 0} \E(\omu(\delta, \cdot)) &= \int_{0}^{1} \varepsilon \frac{\braket{\curve_{ss}}{\Rot(\curve_s)}}{\abs{\curve_s}^{5}} (\braket{\Rot(\curve_s)}{\oeta_{ss}} + \braket{\curve_{ss}}{\Rot(\oeta_s)}) \ds \\
		&\quad - \int_0^1 \frac{5 \varepsilon}{2} \frac{\braket{\curve_{ss}}{\Rot(\curve_s)}^2}{\abs{\curve_s}^{7}} \braket{\curve_s}{\oeta_{s}} + \frac{\braket{\curve_s}{\oeta_{s}}}{\abs{\curve_s}} \ds \\
		&\quad - \braket*{\frac{\curve(1) - \curve(0)}{\abs{\curve(1) - \curve(0)}^2}}{\oeta(1) - \oeta(0)} \\
		&= \int_{0}^{1} \frac{\varepsilon}{\len_\gamma^{3}} \braket{\curve_{ss}}{\oeta_{ss}} - \frac{3 \varepsilon}{2} \frac{\curva_\gamma^{2}}{\len_\gamma} \braket{\curve_s}{\oeta_{s}} + \frac{1}{\len_\gamma} \braket{\curve_s}{\oeta_{s}} \ds. \\
		&\quad -\braket*{\frac{\curve(1) - \curve(0)}{\abs{\curve(1) - \curve(0)}^2}}{\oeta(1) - \oeta(0)} = E(\curve, \oeta).
	\end{align*}

	\textit{First variation of $\D_1$}:
	\begin{equation} \label{David}
		\frac{d}{d\delta} |_{\delta = 0} \D_1(\omu(\delta, \cdot)) = \frac{d}{d\delta} |_{\delta = 0} \D_1(\curve + \delta \oeta) = \braket*{\frac{\curve(0) - \tilde\curve(0)}{\tau}}{\oeta(0)} + \braket*{\frac{\curve(1) - \tilde\curve(1)}{\tau}}{\oeta(1)}.
	\end{equation}

	\textit{First variation of $\D_2$}:
	First note that by comparing \eqref{Myrtle}, \eqref{Dale} and Definition \ref{Bruce} we see that
	\begin{equation*}
		P_\delta(0, \cdot) = \LM_1(\curve, \oeta), \quad P_{s\delta}(0, \cdot) = \LM_2(\curve, \oeta).
	\end{equation*}
	Furthermore, we preliminary compute
	\begin{equation*}
		\omu_{\delta}(\delta, s) = \curve_s(P(\delta, s))P_{\delta}(\delta, s) + \oeta(P(\delta, s)) + \delta \oeta_s(P(\delta, s))P_{\delta}(\delta, s).
	\end{equation*}
	Hence by the dominated convergence Theorem we have
	\begin{equation} \label{Walter}
		\begin{aligned}
			\frac{d}{d\delta} |_{\delta = 0} \D_2(\omu(\delta, \cdot)) &= \frac{1}{2 \len_{\tilde\curve}} \int_{0}^{1} \braket*{\frac{\curve - \tilde\curve}{\tau}}{\Rot(\tilde\curve_s)} \braket{\Rot(\tilde\curve_s)}{\oeta} \ds \\
			&\quad + \frac{1}{2 \len_{\tilde\curve}} \int_{0}^{1} \braket*{\frac{\curve - \tilde\curve}{\tau}}{\Rot(\tilde\curve_s)} \braket{\Rot(\tilde\curve_s)}{\LM_1(\curve, \oeta) \curve_s} \ds
		\end{aligned}
	\end{equation}

	\textit{First variation of $\D_3$}:
	We preliminary compute
	\begin{align*}
		\omu_s(\delta, s) &= \curve_s(P(\delta, s)) P_{s}(\delta, s) + \delta \oeta_{s}(P(\delta, s)) P_{s}(\delta, s), \\
		\begin{split}
			\omu_{s \delta}(\delta, s) &= \curve_{ss}(P(\delta, s)) P_{\delta}(\delta, s) P_{s}(\delta, s) + \curve_s(P(\delta, s)) P_{s \delta}(\delta, s) + \oeta_s(P(\delta, s)) P_{s}(\delta, s) \\
			&\quad + \delta \oeta_{ss}(P(\delta, s)) P_{\delta}(\delta, s) P_{s}(\delta, s) + \delta \oeta_{s}(P(\delta, s)) P_{s \delta}(\delta, s),
		\end{split} \\
		\frac{d}{d \delta} |_{\delta = 0} \frac{1}{\len_{\omu(\delta, \cdot)}} &= - \frac{1}{\len_\gamma^{2}} \int_{0}^{1} \braket*{\frac{\curve_s}{\len_\gamma}}{\oeta_{s}} \ds.
	\end{align*}
	With the computation above and thanks to the dominated convergence Theorem we derive
	\begin{equation} \label{Marian}
		\begin{aligned}
			\frac{d}{d\delta} |_{\delta = 0} \D_3(\omu(\delta, \cdot)) &= \frac{1}{2 \len_\gamma} \int_{0}^{1} \braket*{\frac{\curve - \tilde\curve}{\tau}}{\Rot(\curve_s)} \braket{\Rot(\curve_s)}{\LM_1(\curve, \oeta) \curve_s + \oeta} \ds \\
			&\quad + \frac{1}{2 \len_\gamma} \int_{0}^{1} \braket*{\frac{\curve - \tilde\curve}{\tau}}{\Rot(\curve_s)} \braket{\curve - \tilde\curve}{\LM_1(\curve, \oeta) \Rot(\curve_{ss}) + \LM_2(\curve, \oeta) \Rot(\curve_s) + \Rot(\oeta_s)} \ds \\
			&\quad -\frac{1}{\len_\gamma^{2}} \int_{0}^{1} \braket*{\frac{\curve_s}{\len_\gamma}}{\oeta_{s}} \ds \int_{0}^{1} \frac{1}{4 \tau} \braket{\curve - \tilde\curve}{\Rot(\curve_s)}^{2} \ds
		\end{aligned}
	\end{equation}
	By collecting all the terms in \eqref{David}, \eqref{Walter} and \eqref{Marian}, \eqref{Christopher} and \eqref{Connor} follow.
\end{proof}

\begin{theorem}[Time-discrete geometric evolution] \label{Betty}
	For any $T > 0$ there exists a $\tau_{0} = \tau_{0}(\varepsilon, T) > 0$ such that for every $\oeta \in \C ^\infty(\RR_+, \C ^\infty)$ and every $\tau < \tau_0$ it holds that
	\begin{equation} \label{Ryan}
		\int_{0}^{T} \E(\curve^{\tau}(t, \cdot), \oeta(t, \cdot)) + \D(\curve^{\tau}(t, \cdot), \oeta(t, \cdot)) + \Err(\curve^{\tau}(t, \cdot), \oeta(t, \cdot)) \dt = 0,
	\end{equation}
	where $\E$, $\D$ and $\Err$ are as in \eqref{Lester}, \eqref{Christopher} and \eqref{Connor}, respectively.
\end{theorem}

\begin{proof}[Proof of Theorem \ref{Betty}]
	The proof follows by using Remark \ref{Johanna}, \eqref{Alberta} and a simple induction argument.
\end{proof}

The weak formulation \eqref{Ryan} of the time-discete evolution will now be used to derive further compactness results. We start with

\begin{theorem} \label{Willie}
	Let $(\curve^{\tau})$ and $\curve$ be as in Theorem \ref{Emilie}. Then, up to a subsequence,
	\begin{equation} \label{Glenn}
		\curve^{\tau} \to \curve \text{ in } L^{2}_{\loc}(\mathbb{R}_{+}, H^{2}).
	\end{equation}
\end{theorem}

\begin{proof}[Proof of Theorem \ref{Willie}]
	Fix $T > 0$ and let $\tau_{0}$ be as in Theorem \ref{Betty}. We wish to show that $(\curve^{\tau})$ is a Cauchy-sequence in $L^{2}(0, T; H^{2})$. Now fix $\delta > 0$. Due to \eqref{Jeremiah} there exists $\tau_{1} = \tau_{1}(\delta) > 0$ such that for all $0 < \sigma < \tau < \tau_{1}$ we have for $\Delta \curve := \curve^{\tau} - \curve^{\sigma}$
	\begin{equation} \label{Kate}
		\norm{\Delta \curve}_{L^{\infty}(0, T; \C ^{1})} < \delta.
	\end{equation}
	Let us take $0 < \sigma < \tau < \min\{ \tau_{0}, \tau_{1} \}$. We first write
	\begin{equation} \label{Nell}
		\begin{aligned}
			\frac{\varepsilon}{(L^\tau)^3} \int_{0}^{T} \int_{0}^{1} \abs{\Delta \curve_{ss}}^{2} \ds \dt &= \int_{0}^{T} \int_{0}^{1} \frac{\varepsilon}{(L^{\tau})^3} \braket{\curve^{\tau}_{ss}}{\Delta \curve_{ss}} - \frac{\varepsilon}{(L^{\sigma})^3} \braket{\curve^{\sigma}_{ss}}{\Delta \curve_{ss}} \ds \dt \\
			&\quad + \int_{0}^{T} \int_{0}^{1} \varepsilon \left(\frac{1}{(L^{\sigma})^3} - \frac{1}{(L^{\tau})^3}\right) \braket{\curve^{\sigma}_{ss}}{\Delta \curve_{ss}} \ds \dt.
		\end{aligned}
	\end{equation}
	Subtracting \eqref{Ryan} with time-step $\sigma$ and $\oeta = \Delta \curve$ from \eqref{Ryan} with time-step $\tau$ and again $\oeta = \Delta \curve$ we rewrite \eqref{Nell} as the sum
	\begin{equation} \label{Jane}
		\frac{\varepsilon}{(L^\tau)^3} \int_{0}^{T} \int_{0}^{1} \abs{\Delta \curve_{ss}}^{2} \ds \dt = A + \Bit_1^\sigma - \Bit_1^\tau + \Bit_2^\sigma - \Bit_2^\tau + \Bit_3^\sigma - \Bit_3^\tau,
	\end{equation}
	where
	\begin{align*}
		A &= \int_{0}^{T} \int_{0}^{1} \varepsilon \left(\frac{1}{(L^{\sigma})^3} - \frac{1}{(L^{\tau})^3}\right) \braket{\curve^{\sigma}_{ss}}{\Delta \curve_{ss}} \ds \dt, \\
		\Bit_1^\sigma &= \int_{0}^{T} \int_{0}^{1} \frac{1}{L^\sigma} (1 - \frac{3 \varepsilon}{2} (\curva^{\sigma})^2) \braket{\curve^{\sigma}_{s}}{\Delta \curve_s} \ds \dt - \int_{0}^{T} \braket*{\frac{\curve^{\sigma}(t, 1) - \curve^{\sigma}(t, 0)}{\abs{\curve^{\sigma}(t, 1) - \curve^{\sigma}(t, 0)}^2}}{\Delta \curve(t, 1) - \Delta \curve(t, 0)} \dt, \\
		\Bit_2^\sigma &= \int_{0}^{T} \frac{1}{2 \tilde{L}^{\sigma}} \int_{0}^{1} \braket{\velo^{\sigma}}{\Rot(\tilde\curve^\sigma_s)} \braket{\Rot(\tilde\curve^\sigma_s)}{\Delta \curve} \ds + \frac{1}{2 L^{\sigma}} \int_{0}^{1} \braket{\velo^{\sigma}}{\Rot(\curve^\sigma_s)} \braket{\Rot(\curve^\sigma_s)}{\Delta \curve} \ds \dt \\
		&\quad + \int_{0}^{T} \braket{\velo^{\sigma}(t, 0)}{\Delta \curve(t, 0)} + \braket{\velo^{\sigma}(t, 1)}{\Delta \curve(t, 1)} \dt, \\
		\Bit_3^\sigma &= \int_{0}^{T} \frac{1}{2 L^{\sigma}} \int_{0}^{1} \braket{\velo^{\sigma}}{\Rot(\tilde\curve^\sigma_s)} \braket{\Rot(\tilde\curve^\sigma_s)}{\curve^{\sigma}_{s}} \LM_1(\curve^\sigma, \Delta \curve) \ds \dt \\
		&\quad + \int_{0}^{T} \frac{1}{2 L^{\sigma}} \int_{0}^{1} \braket{\velo^{\sigma}}{\Rot(\tilde\curve^\sigma_s)} \braket{\curve^{\sigma} - \tilde\curve^{\sigma}}{\LM_1(\curve^\sigma, \Delta \curve) \Rot(\curve^\sigma_{ss}) + \LM_2(\curve^\sigma, \Delta \curve) \Rot(\curve^\sigma_s) + \Delta \Rot(\curve_s)} \ds \dt \\
		&\quad -\int_{0}^{T} \frac{1}{4 (L^{\sigma})^3} \int_{0}^{1} \braket{\curve^{\sigma}_{s}}{\Delta \curve_s} \ds \int_{0}^{1} \braket{\velo^{\sigma}}{\Rot(\curve^\sigma_s)} \braket{\curve^{\sigma} - \tilde\curve^{\sigma}}{\Rot(\curve^\sigma_s)} \ds \dt, \\
	\end{align*}
	and $\Bit_i^\tau$, $i \in \{1, 2, 3\}$ are defined by the same formula as $\Bit_i^\sigma$, but with each $\sigma$ exchanged with $\tau$. We wish to bound the right-hand side of \eqref{Jane}. This will be achieved by taking advantage of \eqref{Kate} thanks to which we can bound every $\Delta \curve$ -and $\Delta \curve_s$-term appearing on the right-hand side of \eqref{Jane} by $\delta$ from above. For all the remaining terms it will be enough to find an upper bound $\C(\varepsilon, T) < \infty$ independent of $\tau$ and $\sigma$. More precisely:

	\textit{$A$-term:}
	Since $L^{\sigma}, \, L^{\tau} \geq c > 0$, due to the Lipschitz-continuity of $x \mapsto 1/x^{3}$ away from $0$ we have:
	\begin{align*}
		\abs*{\frac{1}{L^{\sigma 3}} - \frac{1}{L^{\tau 3}}} &\leq \C \abs{L^{\sigma} - L^{\tau}} = \C \abs*{ \int_{0}^{1} \abs{\curve^{\sigma}_{s}} - \abs{\curve^{\tau}_{s}} \ds} \leq \C(\varepsilon) \int_{0}^{1} \abs{\curve^{\sigma}_{s} - \curve^{\tau}_{s}} \ds \leq \C(\varepsilon) \delta.
	\end{align*}
	Combining this with the bound on the curvature in \eqref{Grace} we have
	\begin{equation*}
		\abs{A} \leq \C(\varepsilon) \int_{0}^{T} \int_{0}^{1} \abs{\curve^{\sigma}_{ss}}^{2} + \abs{\curve^{\tau}_{ss}}^{2} \ds \dt \leq \C(\varepsilon) \delta.
	\end{equation*}

	\textit{$\Bit_1^\sigma$-term:} Thanks to \eqref{Bessie} and \eqref{Grace} we have
	\begin{equation*}
		\begin{split}
			\abs{\Bit_1^\sigma} &\leq \C(\varepsilon) \int_{0}^{T} \int_{0}^{1} \left(1 + (\curva^{\sigma})^2 \right) \abs{\Delta \curve_s} \ds \dt + \C \int_{0}^{T} \abs{\Delta \curve(t, 1)} + \abs{\Delta \curve(t, 0)} \dt \\
			&\leq \C(\varepsilon) \delta + \C \delta = \C(\varepsilon) \delta.
		\end{split}
	\end{equation*}

	\textit{$\Bit_2^\sigma$-term:} Due to \eqref{Bessie}, \eqref{Grace} and \eqref{Georgia} we have that
	\begin{equation} \label{Jonathan}
		\begin{aligned}
			\abs{\Bit_2^\sigma} &\leq \C \int_{0}^{T} \int_{0}^{1} (\abs{\velo^{\sigma}} + \abs{\velo^{\tau}}) \abs{\Delta \curve} \ds \dt \\
			&\quad + \int_{0}^{T} \abs{\velo^{\sigma}(t, 0)} \abs{\Delta \curve(t, 0)} + \abs{\velo^{\sigma}(t, 1)} \abs{\Delta \curve(t, 1)} \dt \\
			&\leq \sqrt{T} \delta \left( \int_{0}^{T} \int_{0}^{1} \abs{\velo^{\sigma}}^{2} \dt \right)^{\frac{1}{2}} + \sqrt{T} \delta \left( \int_{0}^{T} \abs{\velo^{\sigma}(t, 0)}^{2} + \abs{\velo^{\sigma}(t, 1)}^{2} \dt \right)^{\frac{1}{2}} \leq \C(\varepsilon, T) \delta
		\end{aligned}
	\end{equation}

	\textit{$\Bit_3^\sigma$-term:}
	The bound on the $\Bit_3^\sigma$-term can be obtained arguing as in \eqref{Jonathan} using \eqref{Bessie}, \eqref{Grace}, \eqref{Georgia} and noticing that $P_i(\curve^\sigma, \Delta \curve)$, $i = 1, 2$ can be bounded as follows:
	\begin{equation*}
		\max_{i = 1, 2} \abs{P_i(\curve^\sigma, \Delta \curve)} \leq \C \int_{0}^{1} \abs{\braket{\curve^{\sigma}_{s}}{\Delta{\curve}_{s}}} \ds \leq \C \delta.
	\end{equation*}
	Similarly one can bound the $\Bit_i^\tau$-terms by $\C \delta$.

	Exploiting again \eqref{Bessie} and taking into account all the previous estimates we eventually get
	\begin{equation} \label{Leila}
		\con(\varepsilon) \int_0^T \int_0^1 \abs{\Delta \curve_{ss}}^2 \ds \dt \leq \frac{\varepsilon}{(L^\tau)^3} \int_0^T \int_0^1 \abs{\Delta \curve_{ss}}^2 \ds \dt \leq \C(\varepsilon, T) \delta.
	\end{equation}
	By \eqref{Kate} and \eqref{Leila} we have that $(\curve^{\tau})$ is a Cauchy-sequence in $L^{2}(0, T; H^{2})$ whose limit being $\curve$ due to \eqref{Jeremiah}.
\end{proof}

\begin{corollary} \label{Gianluca}
	Let $(\dcurve)$ and $\curve$ be as in Theorem \ref{Emilie}. As $\dcurve(\cdot, t) \in \AC$ for all $t$ and by \eqref{Jeremiah} and \eqref{Glenn} we see that $\curve(\cdot, t) \in \AC$ for almost all $t$.
\end{corollary}

We continue by employing the ellipticity of \eqref{Ryan} and a boot-strapping argument in order to show boundedness of higher order $s$-derivatives of $\dcurve$.

\begin{lemma}[Boot-strapping] \label{Cordelia}
	Let $T > 0$ be fixed and $\tau < \tau(\varepsilon, T)$ with $\tau(\varepsilon, T)$ as in Remark \ref{Johanna}. Then $\curve^{\tau}_{ssss}(\cdot, t)$ exists for all $t \in [0, T]$ and
	\begin{equation} \label{Josie}
		\norm{\curve^{\tau}_{ssss}}_{L^{\frac{3}{2}}(0, T; L^{\frac{3}{2}})} \leq \C(\varepsilon, T) < \infty.
	\end{equation}
\end{lemma}

\begin{proof}[Proof of Lemma \ref{Cordelia}]
	Let us, for the moment, fix $t \in [0, T]$ be fixed. In order not to overburden the reader we write $\dcurve(\cdot) := \curve^{\tau}(\cdot, t)$, $\tilde\dcurve(\cdot) := \tilde\curve^{\tau}(\cdot, t)$, $\dcurva(\cdot) := \dcurva(\cdot, t)$, $L^\tau := L^\tau(t)$, $\tilde{L}^\tau := \tilde{L}^\tau(t)$ and $\dvelo(\cdot) := \dvelo(\cdot, t)$.
	Furthermore we define for any $\fit \colon [0, 1] \to \RR^2$:
	\begin{equation*}
		D_s^{-1} \fit(s) := \int_0^s f(\tilde s) \dtildes,
	\end{equation*}
	and $D_s^{-(n+1)}$ recursively as $D_{s}^{-1} D_s^{-n}$. Integrating by parts in \eqref{Lester}, \eqref{Christopher} for a fixed $\oeta \in C^\infty_c(\ui; \RR^2)$ leads to:
	\begin{align}
		\E(\dcurve, \oeta) &= \int_0^1 \braket*{\frac{\varepsilon}{(\dlen)^3} \dcurve_{ss} + D_s^{-1} \Ait_1}{\oeta_{ss}} \ds \label{Eboot} \\
		\D(\dcurve, \oeta) &= \int_0^1 \braket{D_s^{-2} \Ait_2}{\oeta_{ss}} \ds, \label{Dboot}
	\end{align}
	where
	\begin{align*}
		\Ait_1 &:= -\frac{1}{\dlen} (1 - \frac{3 \varepsilon}{2} \curva_\gamma^2) \dcurve_s, \\
		\Ait_2 &:= \frac{1}{2 \tilde \dlen} \braket{\dvelo}{\Rot(\tilde\dcurve_s)} \Rot(\tilde\dcurve_s) + \frac{1}{2 \dlen} \braket{\dvelo}{\Rot(\dcurve_s)} \Rot(\dcurve_s).
	\end{align*}
	Using the definition of $\LM_1$ and $\LM_2$ (see \eqref{def:P1} and \eqref{def:P2}), Fubini's Theorem and integrating by parts in \eqref{Connor} for a fixed $\oeta \in C^\infty_c(\ui; \RR^2)$ we have
	\begin{equation} \label{Errboot}
		\Err(\dcurve, \oeta) = \int_0^1 \braket*{D_s^{-1} \left\{ \left(\int_s^1 \Ait_3 \dtildes - \int_0^1 \tilde s \Ait_3 \dtildes + \Ait_4 - \int_0^1 \Ait_4 \dtildes + \Ait_5 \right) \curve_s + \Ait_6\right\}}{\oeta_{ss}} \ds,
	\end{equation}
	where
	\begin{align*}
		\Ait_3 &:= \frac{1}{2 \tilde \dlen (\dlen)^2} \braket{\dvelo}{\Rot(\tilde \dcurve_s)} (\braket{\Rot(\tilde \dcurve_s)}{\dcurve_s} + \braket{\dcurve - \tilde \dcurve}{\Rot(\dcurve_{ss})}), \\
		\Ait_4 &:= \frac{1}{2 \tilde \dlen (\dlen)^2} \braket{\dvelo}{\Rot(\tilde \dcurve_s)} \braket{\dcurve - \tilde \dcurve}{\Rot(\dcurve_s)}, \\
		\Ait_5 &:= \frac{1}{4 (\dlen)^3} \int_0^1 \braket{\dvelo}{\Rot(\dcurve_s)} \braket{\dcurve - \tilde \dcurve}{\Rot(\dcurve_s)} \ds,\\
		\Ait_6 &:= -\frac{1}{2 \tilde \dlen} \braket{\dvelo}{\Rot(\tilde \dcurve_s)} \Rot(\dcurve - \tilde \dcurve).
	\end{align*}
	By \eqref{Eboot}, \eqref{Dboot}, \eqref{Errboot} and the Euler-Lagrange equation \eqref{Alberta} there exists $v$, $w \in \RR^2$ such that 
	\begin{equation} \label{order2}
		-\frac{\varepsilon}{(\dlen)^3} \dcurve_{ss} = v + w s + D_s^{-1} \Ait_7 + D_s^{-2} \Ait_2,
	\end{equation}
	where
	\begin{equation*}
		\Ait_7 := \Ait_1 + \left(\int_s^1 \Ait_3 \dtildes - \int_0^1 \tilde s \Ait_3 \dtildes + \Ait_4 - \int_0^1 \Ait_4 \dtildes + \Ait_5 \right) \dcurve_s + \Ait_6.
	\end{equation*}
	As the right-hand side of \eqref{order2} is weakly differentiable we can further differentiate $\dcurve_{ss}$ to obtain
	\begin{equation} \label{order3}
		-\frac{\varepsilon}{(\dlen)^3} \dcurve_{sss} = w + \Ait_7 + D_s^{-1} \Ait_2.
	\end{equation}
	By the very definition of $\Ait_7$, thanks to the regularity of $\dcurve$ \eqref{order3} shows that $\dcurve$ is four times weakly differentiable and that
	\begin{equation*}
		-\frac{\varepsilon}{(\dlen)^3} \dcurve_{ssss} = (\Ait_7)_s + \Ait_2.
	\end{equation*}
	For convenience we will now split up the right-hand side of the equation above as follows
	\begin{equation} \label{Jayden}
		-\frac{\varepsilon}{(\dlen)^3} \dcurve_{ssss} = \sum_{i = 1}^5 \Bit_i,
	\end{equation}
	where
	\begin{align*}
		\Bit_1 &:= (\Ait_1)_s = \frac{1}{\dlen} \left(3 \varepsilon \dcurva \dcurva_s \dcurve_{s} - (1 - \frac{3 \varepsilon}{2} (\dcurva)^{2}) \dcurve_{ss}\right), \\
		\Bit_2 &:= A_2, \\
		\Bit_3 &:= \Ait_3 \curve_s^\tau + \left(\int_s^1 \Ait_3 \dtildes - \int_0^1 \tilde s \Ait_3 \dtildes\right) \dcurve_{ss}, \\
		\Bit_4 &:= (\Ait_4)_s \dcurve_{s} + \left(\Ait_4 - \int_{0}^{1} \Ait_4 \dtildes \right) \dcurve_{ss}, \\
		\Bit_5 &:= (\Ait_5)_s \dcurve_{s} + \Ait_5 \dcurve_{ss} + (\Ait_6)_s.
	\end{align*}
	We will now estimate each term on the right-hand side of \eqref{Jayden} separately. We note that we will repeatedly make use of \eqref{Bessie}, \eqref{Grace}, \eqref{Leon} and of the boundedness implied by the convergence in \eqref{Jeremiah}.

	 \textit{$\Bit_1$-term}:
	\begin{equation} \label{Henry}
		\begin{aligned}
			\int_{0}^{1} \abs{\Bit_1}^{\frac{3}{2}} \ds &\leq \C(\varepsilon) \int_{0}^{1} \abs{\dcurva}^{\frac{3}{2}} \abs{\dcurva_{s}}^{\frac{3}{2}} + \abs{\dcurve_{ss}}^{\frac{3}{2}} + \abs{\dcurva}^{3} \abs{\dcurve_{ss}}^{\frac{3}{2}} \ds \\
			&\leq \C(\varepsilon) \int_{0}^{1} \abs{\dcurve_{ss}}^{\frac{3}{2}} (\abs{\dcurve_{sss}}^{\frac{3}{2}} + 1) + \abs{\dcurve_{ss}}^{\frac{9}{2}} \ds \\
			&\leq \C(\varepsilon) \int_{0}^{1} 1 + \abs{\dcurve_{sss}}^{\frac{9}{4}} + \abs{\dcurve_{ss}}^{\frac{9}{2}} \ds,
		\end{aligned}
	\end{equation}
	where in the third line we employed Young's inequality.
	Using the interpolation inequality \eqref{Millie} and eventually Young's inequality with arbitrary $\delta > 0$ leads to
	\begin{equation} \label{Edward}
		\begin{aligned}
			\norm{\dcurve_{sss}}_{L^{\frac{9}{4}}}^{\frac{9}{4}} &\leq \C \left( \norm{\dcurve_{ssss}}_{L^{\frac{3}{2}}}^{\frac{11}{8}} \norm{\dcurve_{ss}}_{L^{\frac{3}{2}}}^{\frac{7}{8}} + \norm{\dcurve_{ss}}_{L^{\frac{3}{2}}}^{\frac{9}{4}}\right). \\
			&\leq \C \left(\delta \norm{\dcurve_{ssss}}_{L^{\frac{3}{2}}}^{\frac{3}{2}} + \C(\delta) \norm{\dcurve_{ss}}_{L^{\frac{3}{2}}}^{\frac{21}{2}} + \norm{\dcurve_{ss}}_{L^{\frac{3}{2}}}^{\frac{9}{4}}\right) \\
			&\leq \C \delta \norm{\dcurve_{ssss}}_{L^{\frac{3}{2}}}^{\frac{3}{2}} + \C(\delta, \varepsilon),
		\end{aligned}
	\end{equation}
	as well as
	\begin{equation} \label{Adele}
		\begin{aligned}
			\norm{\dcurve_{ss}}_{L^{\frac{9}{2}}}^{\frac{9}{2}} &\leq \C \left(\norm{\dcurve_{ssss}}_{L^{\frac{3}{2}}} \norm{\dcurve_{ss}}_{L^{\frac{3}{2}}}^{\frac{7}{2}} + \norm{\dcurve_{ss}}_{L^{\frac{3}{2}}}^{\frac{9}{2}} \right) \\
			&\leq \C \left(\delta \norm{\dcurve_{ssss}}_{L^{\frac{3}{2}}}^{\frac{3}{2}} + \C(\delta) \norm{\dcurve_{ss}}_{L^{\frac{3}{2}}}^{\frac{21}{2}} + \norm{\dcurve_{ss}}_{L^{\frac{3}{2}}}^{\frac{9}{2}} \right) \\
			&\leq \C \delta \norm{\dcurve_{ssss}}_{L^{\frac{3}{2}}}^{\frac{3}{2}} + \C(\delta, \varepsilon).
		\end{aligned}
	\end{equation}
	By \eqref{Henry}, \eqref{Edward} and \eqref{Adele} we get
	\begin{equation*}
		\int_{0}^{1} \abs{\Bit_1}^{\frac{3}{2}} \ds \leq \C(\varepsilon, \delta) + \C(\varepsilon) \delta \norm{\dcurve_{ssss}}_{L^{\frac{3}{2}}}^{\frac{3}{2}}.
	\end{equation*}

	 \textit{$\Bit_2$-term}:
	\begin{equation*}
		\int_{0}^{1} \abs{\Bit_2}^{\frac{3}{2}} \ds \leq \C \norm{\Bit_2}_{L^{2}}^{\frac{3}{2}} \leq \C \norm{\dvelo}_{L^{2}}^{\frac{3}{2}} \leq \C(1 + \norm{\dvelo}_{L^2}^2).
	\end{equation*}

	 \textit{$\Bit_3$-term}:
	\begin{equation} \label{Ronnie}
		\begin{aligned}
			&\int_{0}^{1} \abs{\Bit_3}^{\frac{3}{2}} \ds \leq \C \int_0^1 \left( \abs{\Ait_3}^{\frac{3}{2}} + \left(\int_0^1 \abs{\Ait_3}\dtildes \right)^{\frac{3}{2}} \abs{\dcurve_{ss}}^{\frac{3}{2}} \right) \ds \\
			&\leq \C(T) \int_0^1 \left( \abs{\dvelo}^{\frac{3}{2}}(1 + \abs{\dcurve_{ss}}^{\frac{3}{2}}) + (\norm{\dvelo}_{L^1}^{\frac{3}{2}} + \norm{\dvelo}_{L^2}^{\frac{3}{2}} \norm{\dcurve_{ss}}_{L^2}^{\frac{3}{2}}) \abs{\dcurve_{ss}}^{\frac{3}{2}} \right) \ds \\
			&\leq \C(T) \int_0^1 1 + \C(\delta) \abs{\dvelo}^2 + \delta \abs{\curve^{\tau}_{ss}}^6 \ds + \C(T) (\norm{\dvelo}_{L^1}^{\frac{3}{2}} + \norm{\dvelo}_{L^2}^{\frac{3}{2}} \norm{\dcurve_{ss}}_{L^2}^{\frac{3}{2}}) \norm{\dcurve_{ss}}_{L^{\frac{3}{2}}}^{\frac{3}{2}} \\
			&\leq \C(T, \delta) \left( 1 + \norm{\dvelo}_{L^2}^2 \right) + \C(\varepsilon, T) \delta \norm{\dcurve_{ss}}_{L^6}^6 + \C(\varepsilon, T) \norm{\dvelo}_{L^2}^{\frac{3}{2}} \\
			&\leq \C(\varepsilon, T, \delta) \left( 1 + \norm{\dvelo}_{L^2}^2 \right) + \C(\varepsilon, T) \delta \norm{\dcurve_{ss}}_{L^6}^6,
		\end{aligned}
	\end{equation}
	where in the third line we used Young's inequality with arbitrary $\delta > 0$ to estimate $\abs{\dvelo}^{\frac{3}{2}}(1 + \abs{\dcurve_{ss}}^{\frac{3}{2}})$.
	Making use of the interpolation inequality \eqref{Millie} it follows:
	\begin{equation} \label{Theresa}
		\begin{aligned}
			\norm{\dcurve_{ss}}_{L^{6}}^{6} &\leq \C \left(\norm{\dcurve_{ssss}}_{L^{\frac{3}{2}}}^{\frac{3}{2}} \norm{\dcurve_{ss}}_{L^{\frac{3}{2}}}^{\frac{9}{2}} + \norm{\dcurve_{ss}}_{L^{\frac{3}{2}}}^{6} \right) \\
			&\leq \C \left(\norm{\dcurve_{ssss}}_{L^{\frac{3}{2}}}^{\frac{3}{2}} \norm{\dcurve_{ss}}_{L^{2}}^{\frac{9}{2}} + \norm{\dcurve_{ss}}_{L^{2}}^{6} \right).
		\end{aligned}
	\end{equation}
	Combining \eqref{Ronnie} and \eqref{Theresa} leads to
	\begin{equation*}
		\int_0^1 \abs{\Bit_3}^{\frac{3}{2}} \ds \leq \C(\varepsilon, T, \delta) (1 + \norm{\dvelo}_{L^2}^2) + \C(\varepsilon, T) \delta \norm{\dcurve_{ssss}}_{L^{\frac{3}{2}}}^{\frac{3}{2}}.
	\end{equation*}

	 \textit{$\Bit_4$-term}:
	\begin{equation} \label{Richard}
		\int_0^1 \abs{\Bit_4}^{\frac{3}{2}} \ds \leq \C \int_0^1 \left( \abs{(\Ait_4)_s}^{\frac{3}{2}} + \left( \int_0^1 \abs{\Ait_4} \dtildes \right)^{\frac{3}{2}} \abs{\dcurve_{ss}}^{\frac{3}{2}} + \abs{\Ait_4}^{\frac{3}{2}} \abs{\dcurve_{ss}}^{\frac{3}{2}} \right) \ds.
	\end{equation}
	With
	\begin{equation*}
		\begin{aligned}
			2 \tilde L^\tau (L^\tau)^2 (\Ait_4)_s &= \braket{\dcurve_s - \tilde \dcurve_s}{\Rot(\tilde \dcurve_s)} \braket{\dvelo}{\Rot(\dcurve_s)} + \braket{\dvelo}{\Rot(\tilde \dcurve_{ss})} \braket{\dcurve - \tilde \dcurve}{\Rot(\dcurve_s)} \\
			&\quad + \braket{\dvelo}{\Rot(\tilde \dcurve_s)} \braket{\dcurve_s - \tilde \curve ^\tau_s}{\Rot(\dcurve_s)} + \braket{\dvelo}{\Rot(\tilde \dcurve_s)} \braket{\dcurve - \tilde \curve ^\tau}{\Rot(\dcurve_{ss})}.
		\end{aligned}
	\end{equation*}
	and \eqref{Theresa} we follow
	\begin{equation} \label{Michael}
		\begin{aligned}
			\int_0^1 \abs{(\Ait_4)_s}^{\frac{3}{2}} \ds &\leq \C(T) \int_0^1 \abs{\dvelo}^{\frac{3}{2}} (1 + \abs{\dcurve_{ss}}^{\frac{3}{2}}) \ds \\
			&\leq \C(T) \int_0^1 1 + \C(\delta) \abs{\dvelo}^2 + \delta \abs{\dcurve_{ss}}^6 \ds \\
			&= \C(T, \delta) (1 + \norm{\dvelo}_{L^2}^2) + \C(\varepsilon, T) \delta \norm{\curve^{\tau}_{ssss}}_{L^{\frac{3}{2}}}^{\frac{3}{2}}
		\end{aligned}
	\end{equation}
	Furthermore by \eqref{Theresa} and again Young's inequality we have
	\begin{equation} \label{Justin}
		\begin{aligned}
			\int_0^1 \left( \int_0^1 \abs{\Ait_4} \dtildes \right)^{\frac{3}{2}} \abs{\dcurve_{ss}}^{\frac{3}{2}} + \abs{\Ait_4}^{\frac{3}{2}} \abs{\dcurve_{ss}}^{\frac{3}{2}} \ds &\leq \C(T) \left( \norm{\dvelo}_{L^1}^{\frac{3}{2}} \norm{\dcurve_{ss}}_{L^{\frac{3}{2}}}^{\frac{3}{2}} + \int_0^1 \abs{\dvelo}^{\frac{3}{2}} \abs{\dcurve_{ss}}^{\frac{3}{2}} \ds \right) \\
			&\leq \C(\varepsilon, T) (1 + \norm{\dvelo}_{L^2}^2) + \C(T, \delta) \norm{\dvelo}_{L^2}^{2} + \delta \norm{\dcurve_{ss}}_{L^6}^6 \\
			&\leq \C(\varepsilon, T, \delta) (1 + \norm{\dvelo}_{L^2}^2) + \C(\varepsilon, T) \delta \norm{\dcurve_{ssss}}_{L^{\frac{3}{2}}}^{\frac{3}{2}}.
		\end{aligned}
	\end{equation}
	Combining \eqref{Richard}, \eqref{Michael} and \eqref{Justin} we have
	\begin{equation*}
		\int_0^1 \abs{\Bit_4}^{\frac{3}{2}} \ds \leq \C(\varepsilon, T, \delta) (1 + \norm{\dvelo}_{L^2}^2) + \C(\varepsilon, T) \delta \norm{\dcurve_{ssss}}_{L^{\frac{3}{2}}}^{\frac{3}{2}}.
	\end{equation*}

	 \textit{$\Bit_5$-term}:
	Repeating the same argument as in the previous steps we similarly derive that
	\begin{equation*}
		\int_0^1 \abs{\Bit_5}^{\frac{3}{2}} \ds \leq \C(\varepsilon, T, \delta) (1 + \norm{\dvelo}_{L^2}^2) + \C(\varepsilon, T) \delta \norm{\dcurve_{ssss}}_{L^{\frac{3}{2}}}^{\frac{3}{2}}.
	\end{equation*}

	Hence with \eqref{Jayden} and the bounds we found for the $\Bit_i$-terms we have
	\begin{equation*}
			\con(\varepsilon) \norm{\dcurve_{ssss}}_{L^{\frac{3}{2}}}^{\frac{3}{2}} \leq \C(\varepsilon, T, \delta) \left( 1 + \norm{\dvelo}_{L^2}^{2} \right) + \C(\varepsilon, T) \delta \norm{\dcurve_{ssss}}_{L^{\frac{3}{2}}}^{\frac{3}{2}}.
	\end{equation*}
	Taking $\delta > 0$ small enough this leads to
	\begin{equation*}
		\norm{\dcurve_{ssss}}_{L^{\frac{3}{2}}}^{\frac{3}{2}} \leq \C(\varepsilon, T) \left( 1 + \norm{\dvelo}_{L^2}^{2} \right).
	\end{equation*}
	As $t \in [0, T]$ was arbitrary we have
	\begin{equation} \label{Olive}
		\norm{\curve^{\tau}_{ssss}(\cdot, t)}_{L^{\frac{3}{2}}}^{\frac{3}{2}} \leq \C(\varepsilon, T) \left( 1 + \norm{\velo^{\tau}(\cdot, t)}_{L^2}^{2} \right)
	\end{equation}
	for all $t \in [0, T]$. Integrating \eqref{Olive} over $t \in [0, T]$ and employing \eqref{Georgia} finally leads to \eqref{Josie}.
\end{proof}

The previously derived bound leads to the following compactness result:

\begin{theorem} \label{Hannah}
	Let $\curve^{\tau}$ and $\curve$ be as in Theorem \ref{Emilie}. It holds
	\begin{align}
		\curve^{\tau} &\rightharpoonup \curve \text{ weakly in } L^{\frac{3}{2}}_\loc(\RR_+; W^{4, \frac{3}{2}}) \label{Logan} \\
		\curve^{\tau} &\to \curve \text{ in } L^{\frac{39}{23}}_\loc(\RR_+; W^{3, \frac{39}{23}}). \label{Calvin}
	\end{align}
	In particular we have that for almost all $t$:
	\begin{equation} \label{William}
		\curve^{\tau}(\cdot, t) \to \curve(\cdot, t) \text{ in } \C ^{2}.
	\end{equation}
\end{theorem}

\begin{proof}[Proof of Theorem \ref{Hannah}]
	Let us fix $T > 0$. Then
	\begin{equation*}
		\curve^{\tau} \rightharpoonup \curve \text{ weakly in } L^{\frac{3}{2}}(0, T; W^{4, \frac{3}{2}})
	\end{equation*}
	directly follows from \eqref{Josie}. We will show next that $(\curve^{\tau})$ is Cauchy in $L^{\frac{39}{23}}(0, T; W^{2, \frac{39}{23}})$. Fix $\delta \geq \tilde{\delta} > 0$. From Theorem \ref{Willie} we know that there exists $\tau_0 = \tau_0(\tilde{\delta}) > 0$ small enough such that for any $0 < \sigma < \tau < \tau_0$ and for $\Delta \curve := \curve^{\sigma} - \tilde\curve^{\tau}$ it holds that
	\begin{equation} \label{Alex}
		\norm{\Delta \curve}_{L^{\frac{39}{23}}(0, T; W^{2, \frac{39}{23}})} \leq \C \norm{\Delta \curve}_{L^{2}(0, T; H^2)} < \tilde{\delta},
	\end{equation}
	where in the first bound we used Hölder's inequality. Furthermore using the interpolation inequality \eqref{Millie} we have
	\begin{equation*}
		\norm{\Delta \curve_{sss}}_{L^{\frac{39}{23}}} \leq \C \left( \norm{\Delta \curve_{ssss}}_{L^{\frac{3}{2}}}^{\frac{7}{13}} \norm{\Delta \curve_{ss}}_{L^{\frac{3}{2}}}^{\frac{6}{13}} + \norm{\Delta \curve_{ss}}_{L^{\frac{3}{2}}} \right).
	\end{equation*}
	Hence by Hölder's inequality, Lemma \ref{Cordelia} and \eqref{Alex} we derive for all $0 < \sigma < \tau < \tau_0$
	\begin{equation*}
		\begin{aligned}
			\int_0^T \norm{\Delta \curve_{sss}}_{L^{\frac{39}{23}}}^{\frac{39}{23}} \dt &\leq \C \left( \int_0^T \norm{\Delta \curve_{ssss}}_{L^{\frac{3}{2}}}^{\frac{21}{23}} \norm{\Delta \curve_{ss}}_{L^{\frac{3}{2}}}^{\frac{18}{23}} + \norm{\Delta \curve_{ss}}_{L^{\frac{3}{2}}}^{\frac{39}{23}} \dt \right) \\
			&\leq \C \norm{\Delta \curve_{ssss}}_{L^{\frac{3}{2}}(0, T; L^{\frac{3}{2}})} \norm{\Delta \curve_{ss}}_{L^{2}(0, T; L^2)} + \C(T) \norm{\Delta \curve_{ss}}_{L^{2}(0, T; L^2)}^{\frac{39}{46}} \\
			&\leq \C(\varepsilon, T) (\tilde\delta + \tilde{\delta}^{\frac{39}{46}})
		\end{aligned}
	\end{equation*}
	Therefore for $\tilde{\delta}$ small enough we have for $0 < \sigma < \tau \leq \tau_0$:
	\begin{equation} \label{Arthur}
		\norm{\Delta \curve_{sss}}_{L^{\frac{39}{23}}(0, T; L^{\frac{39}{23}})} < \oeta.
	\end{equation}
	Thanks to \eqref{Alex} and \eqref{Arthur} we conclude \eqref{Calvin} through a diagonal argument similar to the proof of Theorem \ref{Emilie}. Consequently \eqref{William} follows from Sobolev embedding.
\end{proof}

Our last compactness result is derived by the coupling relation \eqref{Don}, which will also lead to the equation satisfied by the tangential component of the velocity of $\curve$.

\begin{theorem} \label{Sally}
	Let $\dcurve$ and $\curve$ be as in Theorem \ref{Emilie} and let $V = \curve_t$. Then, up to a subsequence, it holds
	\begin{align}
		\hat{L}^\tau &\rightharpoonup L \text{ weakly in } H^{1}_{\loc}(\mathbb{R}_+), \label{Beatrice} \\
		\braket{\velo^{\tau}}{\tilde\curve^{\tau}_{s} + \curve^{\tau}_{s}} &\rightharpoonup 2 L \velo^\top \text{ weakly in } L^{\frac{3}{2}}_{\loc}(\mathbb{R}_{+}; W^{1, \frac{3}{2}}), \label{Lulu}
	\end{align}
	where $\hat L^\tau$ is as in \eqref{Rufat}, $L := \Ha(\curve) = \abs{\curve_s}$ and $\velo^\top := \braket*{V}{\frac{\curve_s}{L}}$. Furthermore, for almost all $t \geq 0$ and $s \in \ui$ it holds that
	\begin{equation} \label{Viola}
		\velo^\top_s(s, t) = \len_t(t) + \len(t) \curva(s, t) \velo^{\perp}(s, t),
	\end{equation}
	where $\len_t$ denotes the weak derivative of $L$ and $\velo^\perp := \braket*{V}{\frac{\Rot(\curve_s)}{L}}$.
\end{theorem}

\begin{proof}[Proof of Theorem \ref{Sally}]
	Let us fix $T > 0$. We start by integrating \eqref{Don}, for fixed $t$, over $s \in \ui$. Consequently solving for $\hat{L}^{\tau}_t$ leads to
	\begin{equation*}
		\hat{L}^\tau_t = \frac{1}{\tilde{L}^\tau + L^{\tau}} \left(\braket{\dvelo}{\tilde\dcurve_s + \curve^{\tau}_s} |_{s = 0}^{1} - \int_0^1 \braket{\velo^{\tau}}{\tilde{L}^\tau \tilde{\curva}^\tau \Rot(\tilde\curve^{\tau}_s) + L^\tau \dcurva \Rot(\dcurve_s)} \ds \right).
	\end{equation*}
	Integrating the square of the equation above over $t \in [0, T]$, using \eqref{Bessie}, \eqref{Grace} and \eqref{Georgia} we have
	\begin{equation} \label{Jon}
		\begin{aligned}
			\int_0^T (\hat{L}^\tau_t)^2 \dt &\leq \C \int_0^T \abs{\velo^{\tau}(0, t)}^2 + \abs{\velo^{\tau}(1, t)}^2 \dt + \C \int_0^T \left( \left( \int_0^1 \abs{\tilde\curva^{\tau} \dvelo} \ds \right)^2 + \left( \int_0^1 \abs{\curva^{\tau} \dvelo} \ds \right)^2 \right) \dt \\
			&\leq \C \int_0^T \abs{\velo^{\tau}(0, t)}^2 + \abs{\velo^{\tau}(1, t)}^2 \dt + \C \int_0^T \left( \int_0^1 (\tilde{\curva}^\tau)^2 \ds \int_0^1 \abs{\dvelo}^2 \ds \right) \dt \\
			&\quad + \C \int_0^T \left( \int_0^1 (\dcurva)^2 \ds \int_0^1 \abs{\dvelo}^2 \ds \right) \dt \\
			&\leq \C(\varepsilon) \int_0^T \abs{\velo^{\tau}(0, t)}^2 + \abs{\velo^{\tau}(1, t)}^2 \dt + \C(\varepsilon) \int_0^T \int_0^1 \abs{\dvelo}^2 \ds \dt \leq \C(\varepsilon, T).
		\end{aligned}
	\end{equation}
	Hence $(\hat{L}^{\tau})$ is uniformly bounded in $H^1(0, T)$ and therefore
	\begin{equation*}
		\hat{L}^\tau \rightharpoonup L \text{ weakly in } H^{1}(0, T).
	\end{equation*}
	We next take the absolute value of both sides of \eqref{Don} to the power $\frac{3}{2}$ and integrate over $s \in \ui$ and $t \in [0, T]$. By \eqref{Jon}, \eqref{Georgia}, \eqref{Theresa} and \eqref{Josie} we have
	\begin{equation*}
		\begin{aligned}
			\int_0^T \int_0^1 \abs{ \braket{\velo^{\tau}}{\tilde\curve^{\tau}_{s} + \curve^{\tau}_{s}}_{s} }^{\frac{3}{2}} \ds \dt &\leq \C \int_0^T (\hat{L}^\tau_t)^{\frac{3}{2}} \dt + \C \int_0^T \int_0^1 \abs{\dvelo}^{\frac{3}{2}} \abs{\curve^{\tau}_{ss}}^{\frac{3}{2}} \ds \dt \\
			&\leq \C(T) \norm{\hat{L}^\tau_t}_{L^2}^{\frac{3}{2}} + \int_0^T \norm{\dvelo}_{L^2}^{2} + \norm{\curve^{\tau}_{ss}}_{L^6}^{6} \dt \leq \C(\varepsilon, T),
		\end{aligned}
	\end{equation*}
	where in the third line we have employed Young's inequality. Hence $\braket{\velo^{\tau}}{\tilde\curve^{\tau}_{s} + \curve^{\tau}_{s}}_{s}$ is bounded in $L^{\frac{3}{2}}(0, T; L^{\frac{3}{2}})$ and therefore, up to a subsequence,
	\begin{equation*}
		\braket{\velo^{\tau}}{\tilde\curve^{\tau}_{s} + \curve^{\tau}_{s}} \rightharpoonup \braket{V}{2 \curve_s} = 2 L \velo^\top \text{ weakly in } L^{\frac{3}{2}}(0, T; W^{1, \frac{3}{2}}).
	\end{equation*}
	By a diagonal argument similar to the proof of Theorem \ref{Emilie} and by the reasoning above we have \eqref{Beatrice} and \eqref{Lulu}. Finally equation \eqref{Viola} follows by combining the convergences in \eqref{Beatrice} and \eqref{Lulu} with \eqref{Don}.
\end{proof}

\subsection{Convergence}

In this subsection we derive the equations stated in \eqref{Ricky}. We start by employing the compactness results of the previous subsection in order to pass to the limit $\tau \to 0$ in the weak formulation \eqref{Ryan} of the time-discrete evolution.

\begin{theorem}[Weak form of the geometric evolution] \label{Kenneth}
	Let $\curve$ be as in Theorem \ref{Emilie}. For all $\oeta \in \C ^\infty_c(\RR_+; \C ^\infty)$ it holds that
	\begin{equation} \label{Ann}
		\begin{aligned}
			&\int_0^\infty \int_0^1 \frac{\varepsilon}{L^3} \braket{\curve_{ss}}{\oeta_{ss}} + \frac{1}{L} (1 - \frac{3 \varepsilon}{2} \curva^{2}) \braket{\curve_s}{\oeta_{s}} \ds \dt - \int_0^\infty \braket*{\frac{\curve(1, t) - \curve(0, t)}{\abs{\curve(1, t) - \curve(0, t)}^2}}{\oeta(1, t) - \oeta(0, t)} \dt \\
			&\quad +\int_0^T \int_0^1 \frac{1}{L} \braket{V}{\Rot(\curve_s)} \braket{\Rot(\curve_s)}{\oeta} \ds \dt + \int_0^\infty \braket{V(0, t)}{\oeta(0, t)} + \braket{V(1, t)}{\oeta(1, t)} \dt = 0.
		\end{aligned}
	\end{equation}
\end{theorem}

\begin{proof}[Proof of Theorem \ref{Kenneth}]
	By \eqref{Alberta}, in order to show \eqref{Ann} it is enough to prove the following convergences:
	\begin{align}
		\int_0^T \E(\dcurve, \oeta) \dt &\to \int_0^T \E(\curve, \oeta), \label{Evan} \\
		\int_0^T \D(\dcurve, \oeta) \dt &\to \int_0^T \int_0^1 \frac{1}{L} \braket{V}{\Rot(\curve_s)} \braket{\Rot(\curve_s)}{\oeta} \ds \dt \nonumber \\
		&\quad +\int_0^T \braket{V(0, t)}{\oeta(0, t)} + \braket{V(1, t)}{\oeta(1, t)} \dt \label{Gregory}\\
		\int_1^\infty \Err(\dcurve, \oeta) \dt &\to 0, \label{Milton}
	\end{align}
	where $(\dcurve)$ is as in Theorem \ref{Emilie} and $E$, $D$ and $Err$ are defined in \eqref{Lester}, \eqref{Christopher} and \eqref{Connor}, respectively. In the following let $T > 0$ such that $\supp(\oeta) \subset [0, T] \times \ui$.

	\textit{Proof of \eqref{Evan}}:
	By \eqref{Bessie} and the convergence in \eqref{Jeremiah} we see that
	\begin{equation} \label{Maurice}
		\begin{aligned}
			&\abs*{\braket*{\frac{\dcurve(1, t) - \dcurve(0, t)}{\abs{\dcurve(1, t) - \dcurve(0, t)}}} {\oeta(1, t) - \oeta(0, t)} - \braket*{\frac{\curve(1, t) - \curve(0, t)}{\abs{\curve(1, t) - \curve(0, t)}}} {\oeta(1, t) - \oeta(0, t)}} \\
			&\quad \leq \C \abs{\dcurve(1, t) - \dcurve(0, t) - \curve(1, t) + \curve(0, t)} \norm{\oeta}_{L^\infty(0, T; L^\infty)} \\
			&\quad \leq \C(\oeta) \norm{\dcurve - \curve}_{L^\infty(0, T; L^\infty)} \to 0 \text{ as } \tau \to 0.
		\end{aligned}
	\end{equation}
	Employing \eqref{Bessie} and \eqref{Grace} we have
	\begin{equation} \label{Harriett}
		\int_0^1 \frac{\varepsilon}{(L^\tau)^3} \braket{\curve_{ss}^\tau}{\oeta_{ss}} + \frac{1}{L^\tau} (1 - \frac{3 \varepsilon}{2} (\dcurva)^2) \braket{\curve_s^\tau}{\oeta_s} \ds \leq \C(\varepsilon, \oeta).
	\end{equation}
	Hence by \eqref{William}, the dominated convergence Theorem and \eqref{Maurice} we derive \eqref{Evan}.

	\textit{Proof of \eqref{Gregory}}:
	By \eqref{Jeremiah} we have that
	\begin{align*}
		\frac{\braket{\Rot(\tilde\curve^{\tau}_s)}{\oeta}}{\tilde{L}^\tau} \Rot(\tilde\curve^{\tau}_s) &\to \frac{\braket{\Rot(\curve_s)}{\oeta}}{L} \Rot(\curve_s), \\
		\frac{\braket{\Rot(\dcurve_s)}{\oeta}}{L^\tau} \Rot(\dcurve_s) &\to \frac{\braket{\Rot(\curve_s)}{\oeta}}{L} \Rot(\curve_s)
	\end{align*}
	strongly in $L^\infty(0, T; L^\infty)$ and therefore also strongly in $L^2(0, T; L^2)$. Hence by weak-strong convergence we derive
	\begin{equation} \label{Bertie}
		\begin{aligned}
			&\int_0^T \int_0^1 \braket{\dvelo}{ \frac{\braket{\Rot(\tilde\curve^{\tau}_s)}{\oeta}}{2\tilde{L}^\tau} \Rot(\tilde\curve^{\tau}_s) } + \braket{\dvelo}{ \frac{\braket{\Rot(\dcurve_s)}{\oeta}}{2L^\tau} \Rot(\dcurve_s) } \ds \dt \\
			&\to \int_0^T \int_0^1 \frac{1}{L} \braket{V}{\Rot(\curve_s)} \braket{\Rot(\curve_s)}{\oeta} \ds \dt.
		\end{aligned}
	\end{equation}
	Therefore, by additionally using \eqref{Cory}, we conclude the proof of \eqref{Gregory}.

	\textit{Proof of \eqref{Milton}}:
	From \eqref{Jeremiah} and the Definition \ref{Bruce} of $\LM_1$ and $\LM_2$ we derive that
	\begin{equation*}
		\begin{aligned}
			\LM_1(\dcurve, \oeta) &\to \LM_1(\curve, \oeta), \\
			\LM_2(\dcurve, \oeta) &\to \LM_2(\curve, \oeta),
		\end{aligned}
	\end{equation*}
	as $\tau \to 0$ in $L^\infty(0, T; L^\infty)$. Hence \eqref{Jeremiah} and \eqref{Glenn} imply
	\begin{align*}
		\begin{aligned}
			\frac{1}{\tilde{L}^\tau} \braket{\Rot(\tilde\curve^{\tau}_s)}{\LM_1(\dcurve, \oeta) \Rot(\dcurve_s)} \Rot(\tilde\curve^{\tau}_s) &\to 0, \\
			\frac{1}{L} \braket{\dcurve - \tilde\curve^{\tau}}{\LM_2(\dcurve, \oeta) \Rot(\dcurve_s) + \LM_1(\dcurve, \oeta) \Rot(\dcurve_{ss}) + \Rot(\oeta_s)} \Rot(\dcurve_s) &\to 0, \\
			\frac{1}{L^3} \int_0^1 \braket{\dcurve_s}{\oeta_s} \dtildes \braket{\curve^{\tau} - \tilde\dcurve}{\Rot(\dcurve_s)} \Rot(\dcurve_s) &\to 0,
		\end{aligned}
	\end{align*}
	strongly in $L^2(0, T; L^2)$. Therefore as in the previous step the result follows by the weak-strong convergence.
\end{proof}

\begin{corollary} \label{cor:H4reg}
	The time continuous evolution $\curve$ from Theorem \ref{Emilie} satisfies
	\begin{equation} \label{H4reg}
		\curve \in L^2_\loc(\RR_+; H^4).
	\end{equation}
\end{corollary}

\begin{proof}[Proof of Corollary \ref{cor:H4reg}]
	A priori, from the bound in \eqref{Josie} we can only derive that
	\begin{equation*}
		\curve \in L^{\frac{2}{2}}_\loc(\RR_+; W^{4, \frac{3}{2}}).
	\end{equation*}
	In order to improve the integrability from $\frac{3}{2}$ to $2$ we need to repeat the strategy of the proof of Lemma \ref{Cordelia} in the time-continuous setting. Instead of \eqref{Alberta} we will employ \eqref{Ann}. The main difference between those two is the absence of all Lagrange-multiplier terms contained in the $\Err$-term, which vanish in the limit $\tau \to 0$. This improves the resulting a priori bound. As the argument is parallel to the one in the proof of Lemma \ref{Cordelia} we postpone it to the appendix.
\end{proof}

	Due to the higher regularity of $\curve$ derived in Theorem \ref{Cordelia} we will be able to integrate by parts in equation \eqref{Ann}, which will lead to \eqref{Ricky}.

\begin{theorem}[Strong form of the time-continuous evolution] \label{Adelina}
	Let $\curve$ be as in Theorem \ref{Emilie}, then
	\begin{gather*}
		\curve \in C^{0, \beta}_\loc(\RR_+; C^{1, \alpha}) \cap H^1_\loc(\RR_+; L^2) \cap L^2_\loc(\RR_+; H^4), \\
		\velo^\top \in L^{\frac{3}{2}}_\loc(\RR_+; W^{1, \frac{3}{2}}), \\
		\curve(0, \cdot), \, \curve(1, \cdot) \in H^1_\loc(\RR_+; \RR^2),
	\end{gather*}
	where $\alpha \in (0, \frac{1}{2})$ and $\beta \in (0, \frac{1 - 2\alpha}{8})$. Furthermore, for almost all $s \in \ui$ and $t \in \mathbb{R}_+$ $\curve$ satisfies
	\begin{align}
		\velo^\perp(s, t) &= \curva(s, t) - \varepsilon \left(\frac{1}{L^2(t)} \curva_{ss}(s, t) + \frac{1}{2} \curva^{3}(s, t)\right), \label{Angel} \\
		\velo^\top_s(s, t) &= \len_t(t) + \len(t) \curva(s, t) \velo^{\perp}(s, t), \label{Alma} \\
		V(0, t) &= -\frac{\curve(1, t) - \curve(0, t)}{\abs{\curve(1, t) - \curve(0, t)}^2} + \frac{1}{\len(t)} \curve_s(0, t) - \frac{\varepsilon}{L^2(t)} \curva_s(0, t) \Rot(\curve_s)(s, t), \label{May} \\
		V(1, t) &= \phantom{+} \frac{\curve(1, t) - \curve(0, t)}{\abs{\curve(1, t) - \curve(0, t)}^2} - \frac{1}{\len(t)} \curve_s(1, t) + \frac{\varepsilon}{L^2(t)} \curva_s(1, t) \Rot(\curve_s)(s, t). \label{Adrian}
	\end{align}
	Furthermore for almost all $t \in \RR_+$ the following natural boundary conditions hold true:
	\begin{equation} \label{Alta}
		\curva(0, t) = \curva(1, t) = 0.
	\end{equation}
\end{theorem}

\begin{proof}[Proof of Theorem \eqref{Adelina}]
	The regularity statement directly follows from \eqref{Leona}, \eqref{Cody}, \eqref{Cory} and \eqref{H4reg}. Furthermore we note that \eqref{Alma} was already proved in Theorem \ref{Sally} (see \eqref{Viola}). We want to test \eqref{Ann} with $\oeta = \ophi \Rot(\curve_s$) for a scalar function $\ophi \in \C ^\infty_c(\RR_+; \C ^\infty(\ui))$. As $\curve$ is in general not smooth we need to argument by approximation: Due to \eqref{Josie} we can find a sequence $(\omu^n) \subset \C ^\infty_c(\RR_+; \C ^\infty) \cap L^{\frac{3}{2}}_c(\RR_+; W^{4, \frac{3}{2}})$ such that
	\begin{equation*}
		\omu^n \to \curve \text{ strongly in } L^{\frac{3}{2}}_c(\RR_+; W^{4, \frac{3}{2}}),
	\end{equation*}
	so in particular
	\begin{equation} \label{Eva}
		\omu^n \to \curve \text{ strongly in } L^{\frac{3}{2}}_c(\RR_+; \C ^3).
	\end{equation}
	Furthermore by \eqref{Jeremiah} we can also assume that
	\begin{equation} \label{Sara}
		\omu^n \to \curve \text{ strongly in } L^\infty_c(\RR_+; \C ^1).
	\end{equation}
	Let $T > 0$ such that $\supp \ophi \subset [0, T] \times \ui$. By \eqref{Eva}, the definition of $E$, Hölder's inequality, \eqref{Bessie} and \eqref{Theresa} we derive
	\begin{align*}
		&\abs*{\int_0^\infty E(\curve, \ophi \Rot(\omu^n_s)) \dt - \int_0^\infty E(\curve, \ophi \Rot(\curve_s)) \dt} \\
		&\leq \int_0^T \int_0^1 \C(\ophi)(1 + \abs{\curve_{ss}} + \abs{\curve_{ss}}^2) \norm{\omu^n(\cdot, t) - \curve(\cdot, t)}_{\C ^3} \ds \dt \\
		&\leq \C(\ophi) \left(\int_0^T 1 + \norm{\curve_{ss}(\cdot, t)}_{L^6}^6 \dt \right)^{\frac{1}{3}} \left( \int_0^T \norm{\omu^n(\cdot, t) - \curve(\cdot, t)}_{\C ^3}^{\frac{3}{2}} \dt \right)^{\frac{2}{3}} \\
		&\leq \C(\ophi, \varepsilon, T) \norm{\omu^n - \curve}_{L^{\frac{3}{2}}(0, T; \C ^3)} \overset{n \to \infty}{\to} 0.
	\end{align*}
	Furthermore by \eqref{Bessie} and \eqref{Sara} we also have
	\begin{align*}
		&\abs*{ \int_0^\infty \int_0^1 \velo^\perp \braket{\Rot(\curve_s)}{\ophi \Rot(\omu^n_s)} \ds \dt - \int_0^\infty \int_0^1 \velo^\perp \braket{\Rot(\curve_s)}{\ophi \Rot(\omu^n_s)} \ds \dt } \\
		&\leq \C(\ophi) \norm{\velo^\perp}_{L^1(0, T; L^1)} \norm{\omu^n - \curve}_{L^\infty(0, T; \C ^1)} \overset{n \to \infty}{\to} 0.
	\end{align*}
	In a similar fashion we can also derive
	\begin{align*}
		&\left| \int_0^\infty \braket{V(0, t)}{\ophi(0, t) \Rot(\omu^n_s)} + \braket{V(0, t)}{\ophi(0, t) \Rot(\omu^n_s)} \dt \right. \\
		&\quad \left. - \int_0^\infty \braket{V(0, t)}{\ophi(0, t) \Rot(\curve_s)} - \braket{V(0, t)}{\ophi(0, t) \Rot(\curve_s)} \dt \right| \overset{n \to \infty}{\to} 0.
	\end{align*}
	Therefore testing \eqref{Ann} with $\oeta = \ophi \Rot(\omu^n_s)$ and passing to the limit $n \to \infty$ we have
	\begin{equation} \label{Etta}
		\begin{aligned}
			&\int_0^\infty \int_0^1 \varepsilon \curva \ophi_{ss} - \varepsilon L^2 \curva^3 \ophi - L^2 (\curva - \frac{3 \varepsilon}{2} \curva^3) \ophi + L^2 \velo^\perp \ophi \ds \dt \\
			&\quad -\int_0^T \braket*{\frac{\curve(1, t) - \curve(0, t)}{\abs{\curve(1, t) - \curve(0, t)}^2}} {\ophi(1, t) \Rot(\curve_s)(1, t) - \ophi(0, t) \Rot(\curve_s)(0, t)} \dt \\
			&\quad +\int_0^T L \ophi(0, t) \velo^\perp(0, t) + L \ophi(1, t) \velo^\perp(1, t) = 0
		\end{aligned}
	\end{equation}
	Assuming additionally $\ophi(0, t) = \ophi(1, t) = 0$ for all $t \in \RR_+$ and integrating by parts in \eqref{Etta} then leads to
	\begin{equation*}
		\int_0^\infty \int_0^1 (\frac{\varepsilon}{L^2} \curva_{ss} + \frac{\varepsilon}{2} \curva^3 - \curva + \velo^\perp) L^2 \ophi \ds \dt = 0.
	\end{equation*}
	By the arbitrariness of $\ophi$ we see that \eqref{Angel} holds true. Let us now take $\ophi \in \C ^\infty_c(\RR_+; \C ^\infty(\ui))$, possibly nonzero at $\partial \ui$. Integrating by parts in \eqref{Etta} and using \eqref{Angel} we derive that

	\begin{equation*} \label{Timothy}
		\begin{aligned}
			&\int_0^T \varepsilon \curva \ophi_s - \varepsilon \curva_s \ophi |_{s = 0}^{1} \dt \nonumber - \int_0^T \braket*{\frac{\curve(1, t) - \curve(0, t)}{\abs{\curve(1, t) - \curve(0, t)}^2}} {\ophi(1, t) \Rot(\curve_s)(1, t) - \ophi(0, t) \Rot(\curve_s)(0, t)} \dt \nonumber \\
			&\quad + \int_0^T \braket{V(0, t)}{\ophi(0, t) \Rot(\curve_s)(0, t)} + \braket{V(1, t)}{\ophi(1, t) \Rot(\curve_s)(1, t)} \dt = 0
		\end{aligned}
	\end{equation*}
	Choosing $\ophi$ such that $\ophi(\cdot, 0) = \ophi(\cdot, 1) = \ophi_s(\cdot, 1) = 0$ in \eqref{Timothy} leads to
	\begin{equation*}
		\int_0^T \varepsilon \curva(0, t) \ophi_s(0, t) \dt = 0,
	\end{equation*}
	and due to the arbitrariness of $\ophi_s(0, \cdot)$ to
	\begin{equation*}
		\curva(0, t) = 0
	\end{equation*}
	for almost all $t$. In a similar fashion one can derive the same natural boundary condition at $s = 1$ and \eqref{Alta} follows. Employing \eqref{Alta} in \eqref{Timothy} and taking $\ophi$ with $\ophi(1, \cdot) = 0$ leads to
	\begin{equation*}
		\begin{aligned}
			&\int_0^T \varepsilon \curva_s(0, t) \ophi(0, t) + \braket*{\frac{\curve(1, t) - \curve(0, t)}{\abs{\curve(1, t) - \curve(0, t)}^2}} {\Rot(\curve_s)(0, t)} \ophi(0, t) \dt \\
			&\quad + \int_0^\infty \braket{V(0, t)}{\Rot(\curve_s)(0, t)} \ophi(0, t) \dt = 0.
		\end{aligned}
	\end{equation*}
	Hence, by the arbitrariness of $\ophi(0, \cdot)$, we have for almost all $t \in \RR_+$
	\begin{equation} \label{Patrick}
		\velo^\perp(0, t) = -\braket*{\frac{\curve(1, t) - \curve(0, t)}{\abs{\curve(1, t) - \curve(0, t)}^2}}{\frac{\Rot(\curve_s)}{L}} - \varepsilon \frac{1}{L} \curva_s(0, t).
	\end{equation}
	We next test \eqref{Ann} with $\oeta = \ophi \omu^n_s$, where $\ophi \in \C ^\infty_c(\RR_+; \C ^\infty)$ with $\ophi(1, \cdot) \equiv 0$. Passing to the limit $n \to \infty$ as was done previously then leads to
	\begin{align*}
		0 &= \int_0^\infty \int_0^1 \frac{\varepsilon}{L^3} \braket{L \curva \Rot(\curve_s)}{(2L \curva \ophi_s + L \curva_s \ophi) \Rot(\curve_s)} + \frac{1}{L} (1 - \frac{3 \varepsilon}{2} \curva^2) \braket{\curve_s}{\ophi_s \curve_s} \ds \dt \\
		&\quad - \int_0^\infty \braket*{\frac{\curve(1, t) - \curve(0, t)}{\abs{\curve(1, t) - \curve(0, t)}^2}}{- \ophi(0, t) \curve_s(0, t)} \dt + \int_0^\infty \braket{V(0, t)}{\ophi(0, t) \curve_s(0, t)} \dt \\
		&= \int_0^\infty \int_0^1 \left((\frac{\varepsilon}{2} \curva^2 + 1)\ophi_s + \varepsilon \curva \curva_s \ophi\right) L \ds \dt + \int_0^\infty \left(\braket*{\frac{\curve(1, t) - \curve(0, t)}{\abs{\curve(1, t) - \curve(0, t)}^2}}{\frac{\curve_s}{L}} + \velo^\top(0, t)\right) \ophi(0, t) L \dt \\
		&= \int_0^\infty \left( -1 + \braket*{\frac{\curve(1, t) - \curve(0, t)}{\abs{\curve(1, t) - \curve(0, t)}^2}}{\frac{\curve_s}{L}} + \velo^\top(0, t) \right) \ophi(0, t) L \dt.
	\end{align*}
	Due to the arbitrariness of $\ophi(0, t)$ we have
	\begin{equation} \label{John}
		\velo^\top(0, t) = 1 - \braket*{\frac{\curve(1, t) - \curve(0, t)}{\abs{\curve(1, t) - \curve(0, t)}^2}}{\frac{\curve_s}{L}}
	\end{equation}
	for almost every $t \in \RR_+$. By \eqref{Patrick} and \eqref{John} equation \eqref{May} follows. The proof of \eqref{Adrian} works similarly.
\end{proof}

\section{Numerical experiments} \label{Elmer}
In this section we present some numerical experiments with the aim of showing different examples of the curve-shortening evolution derived in the previous sections. In order to make numerical computations we will discretize our curves as it is customary in the framework of discrete differential geometry, see also \cite{bobenko2008discrete}. Hereby, a discrete curve in $\RR^2$ is defined as a finite sequence of $N$ points $\mathbf{x} = (x_i)_{i = 1}^N \subset \RR^2$. They define a zig-zag curve build up from the edges $E_i = [x_i, x_{i + 1}]$, where $1 \leq i < N$. In this framework the constant speed constraint, as used in the previous sections, has the following discrete counterpart: There exists a $l \geq 0$ such that
\begin{equation*}
	\abs{x_{i+1} - x_i} = \const = l \text{ for all } 1 \leq i < N.
\end{equation*}
Hence given $N \in \NN$ we will consider the following set of admissible discrete curves:
\begin{equation*}
	\AC_N^\discr = \{ \mathbf{x} = (x_1, \dots, x_N) \subset \RR^2 \colon \exists \, l \geq 0 \text{ s.t. } \abs{x_{i + 1} - x_i} = l \text{ for all } 1 \leq i < N\}
\end{equation*}
For any $i \in {1, \dots, N-1}$ we define the discrete unit tangent vector $\tau_i$ and normal vector $\nu_i$ as
\begin{align*}
	\tau_i &= \frac{x_{i + 1} - x_i}{\abs{x_{i + 1} - x_i}} = \frac{1}{l} (x_{i + 1} - x_i), \\
	\nu_i &= \Rot(\tau_i).
\end{align*}
For any $i = 1, \dots, N-1$ let $\alpha_i \in [0, \pi]$ be the unique angle between $\tau_{i - 1}$ and $\tau_i$, this means it satisfies $\cos(\alpha_i) = \braket{\tau_{i - 1}}{\tau_i}$. With this, we can define the discrete curvature as
\begin{equation*}
	\curva_i = \frac{2}{l} \tan\frac{\alpha_i}{2} = \frac{2}{l} \frac{\sin(\alpha_i)}{1 + \cos(\alpha_i)} = \frac{2}{l} \frac{\tau_{i - 1} \times \tau_i}{1 + \braket{\tau_{i - 1}}{\tau_i}},
\end{equation*}
Note that $\times$ is the cross product in $\RR^2$ defined as
\begin{equation*}
	v \times w = v_1 w_2 - v_2 w_1,
\end{equation*}

for any $v, w \in \RR^2$. Let us fix $\varepsilon > 0$, $\tau > 0$ and $N \in \NN$. The discrete version of the energy functional in \eqref{Louisa} is $\F^\discr \colon \AC_N^\discr \times \AC_N^\discr \to \RR$ defined as
\begin{align*}
	\F^\discr(\mathbf{x}, \mathbf{\tilde x}) &= - \log \abs{x_1 - x_N} + N l + \frac{\varepsilon l}{2} \sum_{i = 2}^{N-1} \curva_i^2 + \frac{l}{4 \tau} \sum_{i = 1}^{N-1} \braket{x_i - \tilde{x}_i}{\tilde{\nu}_i}^2 + \frac{l}{4 \tau} \sum_{i = 1}^{N-1} \braket{x_i - \tilde{x}_i}{\nu_i}^2 \\
	&\quad + \frac{1}{2\tau} \abs{x_1 - \tilde{x}_1}^2 + \frac{1}{2\tau} \abs{x_N - \tilde{x}_N}^2.
\end{align*}

We can now describe the discrete in space minimizing movement scheme. Let $\mathbf{x}^{(0)} = (x_1^{(0)}, \dots, x_N^{(0)}) \in \AC_N^\discr$. Then, $\mathbf{x}^{(1)}$ is defined as
\begin{equation*}
	\mathbf{x}^{(1)} \in \argmin_{\mathbf{x} \in \AC_N^\discr} \F^\discr(\mathbf{x}, \mathbf{x}^{(0)})
\end{equation*}
and continue by defining $\mathbf{x}^{(2)}$, $\mathbf{x}^{(3)}$, \dots iteratively, similarly to the space continuous setting.

In the following we will show plots of space discrete minimizing movements for several different initial curves. We start with with a curve $\mathbf{x}^{(0)}$ discretizing the graph of the sinus-function restricted on the interval $[-\pi, \pi]$. Figure \ref{fig:sin} shows several steps of the minimizing movement. The coloring of the curves is used to clarify the temporal ordering of the curves. The curves close to the start are violet, while the curves close to the end are red. One can see the straightening motion of the sinus-curve. In the limit $k \to \infty$ the curve converges towards a straight line with unit length.

\begin{figure}[H]
	\centering
	\includegraphics[width=.49\textwidth]{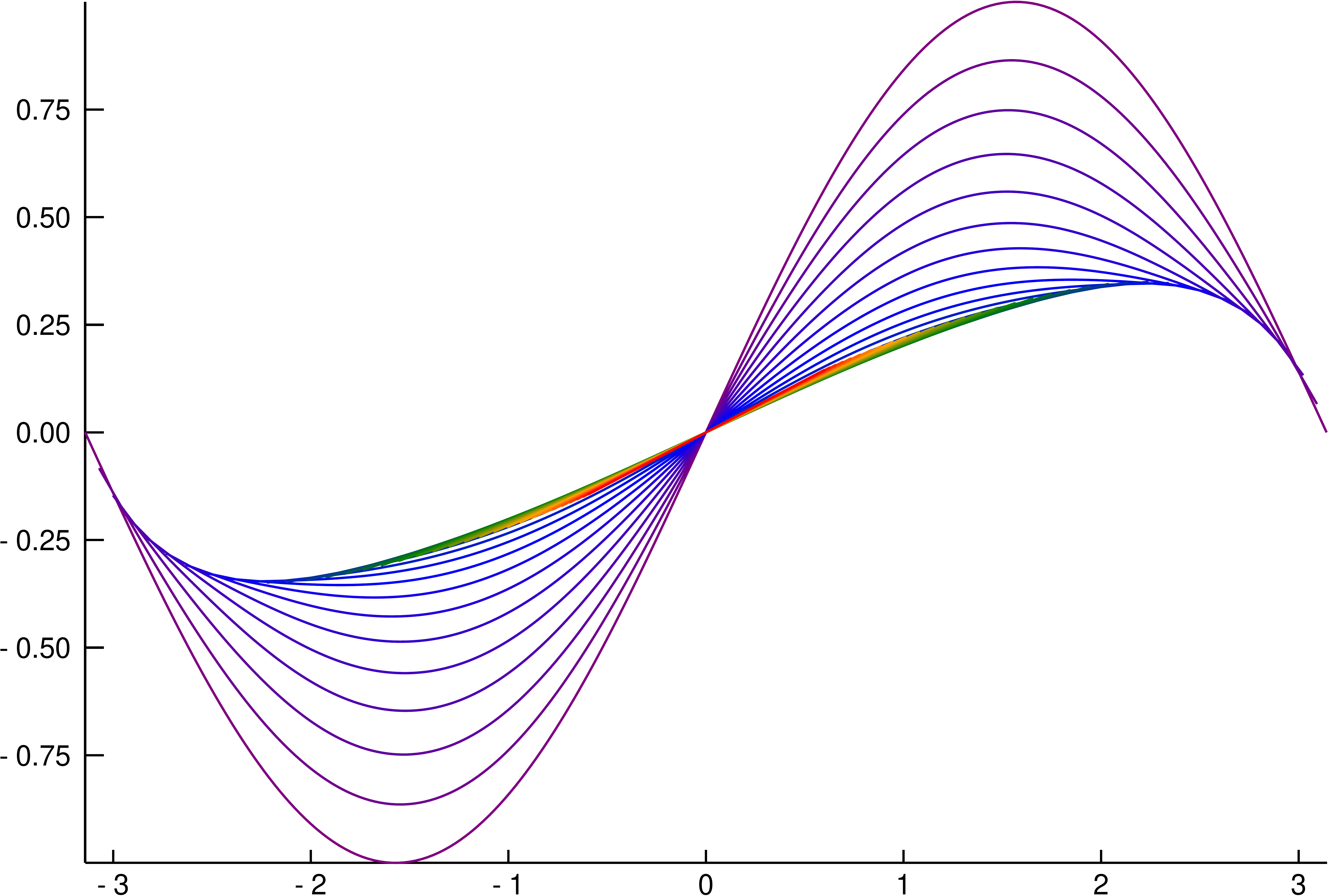}
	\caption{Evolution of a sinus-shaped curve with $\tau = 0.25$ and $\varepsilon = 0.01$.}
	\label{fig:sin}
\end{figure}

We next consider a curve in the shape of a $\gamma$. In figure \ref{fig:gamma_1} one can see that the curl of the curve shrinks until a point where the curvature term becomes dominant. As $k \to \infty$ the curve doesn't unfold and converges towards an "optimal" $\gamma$-shaped curve.
This is a good opportunity to show the dependence of the flow on the size of $\varepsilon$. In figure \ref{fig:gamma_2} we computed the minimizing movement of the $\gamma$-shaped curve from before with a smaller $\varepsilon = 0.01$. One can see that for smaller values of $\varepsilon$ (smaller curvature regularization) the "size" of the curl of the limit curve is smaller.

\begin{figure}[H]
	\centering
	\begin{subfigure}{.49\textwidth}
		\includegraphics[width=\textwidth]{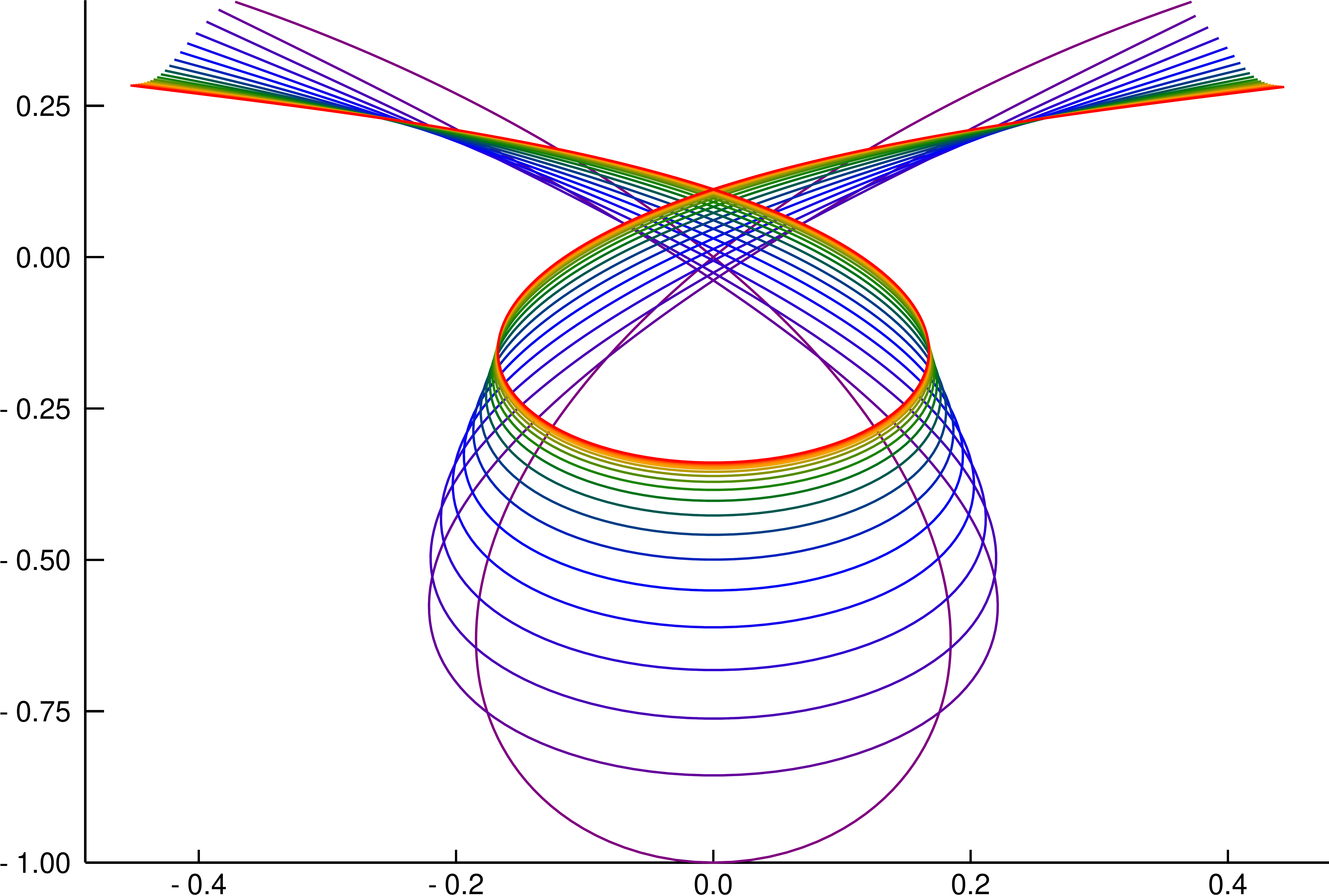}
		\caption{$\varepsilon = 0.1$.}
		\label{fig:gamma_1}
	\end{subfigure}
	\begin{subfigure}{.49\textwidth}
		\includegraphics[width=\textwidth]{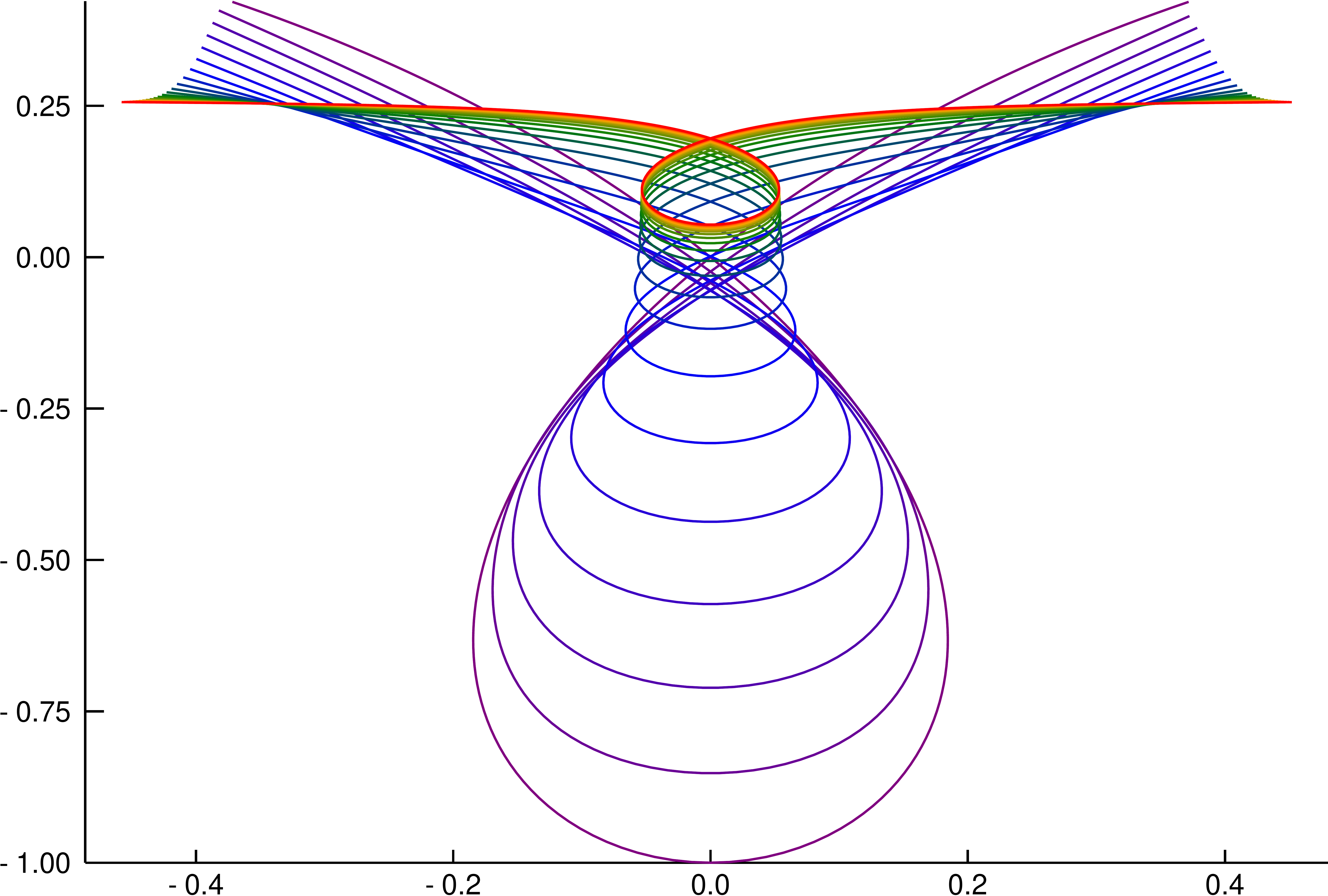}
		\caption{$\varepsilon = 0.01$.}
		\label{fig:gamma_2}
	\end{subfigure}
	\caption{Evolutions of a $\curve$-shaped curve with $\tau = 0.0125$.}
	\label{fig:gamma}
\end{figure}

The previous example showed a curve with self-intersection that didn't unfold during the flow. This is not generally true, as can be seen in the next example in figure \ref{fig:asym_gamma}. In this case one takes an asymmetric $\gamma$-shaped curve as initial curve, which is shorter on the left endpoint and whose evolution can be described as follows: First (see figure \ref{fig:asym_gamma_1}) the curl shrinks until it reaches an optimal size. Then (see figure \ref{fig:asym_gamma_2}) the curl slowly moves to the left endpoint and lastly (see figure \ref{fig:asym_gamma_3}) the curve unfolds at the left endpoint and the converges towards a segment of unit length.

\begin{figure}[H]
	\centering
	\begin{subfigure}{.49\textwidth}
		\includegraphics[width=\textwidth]{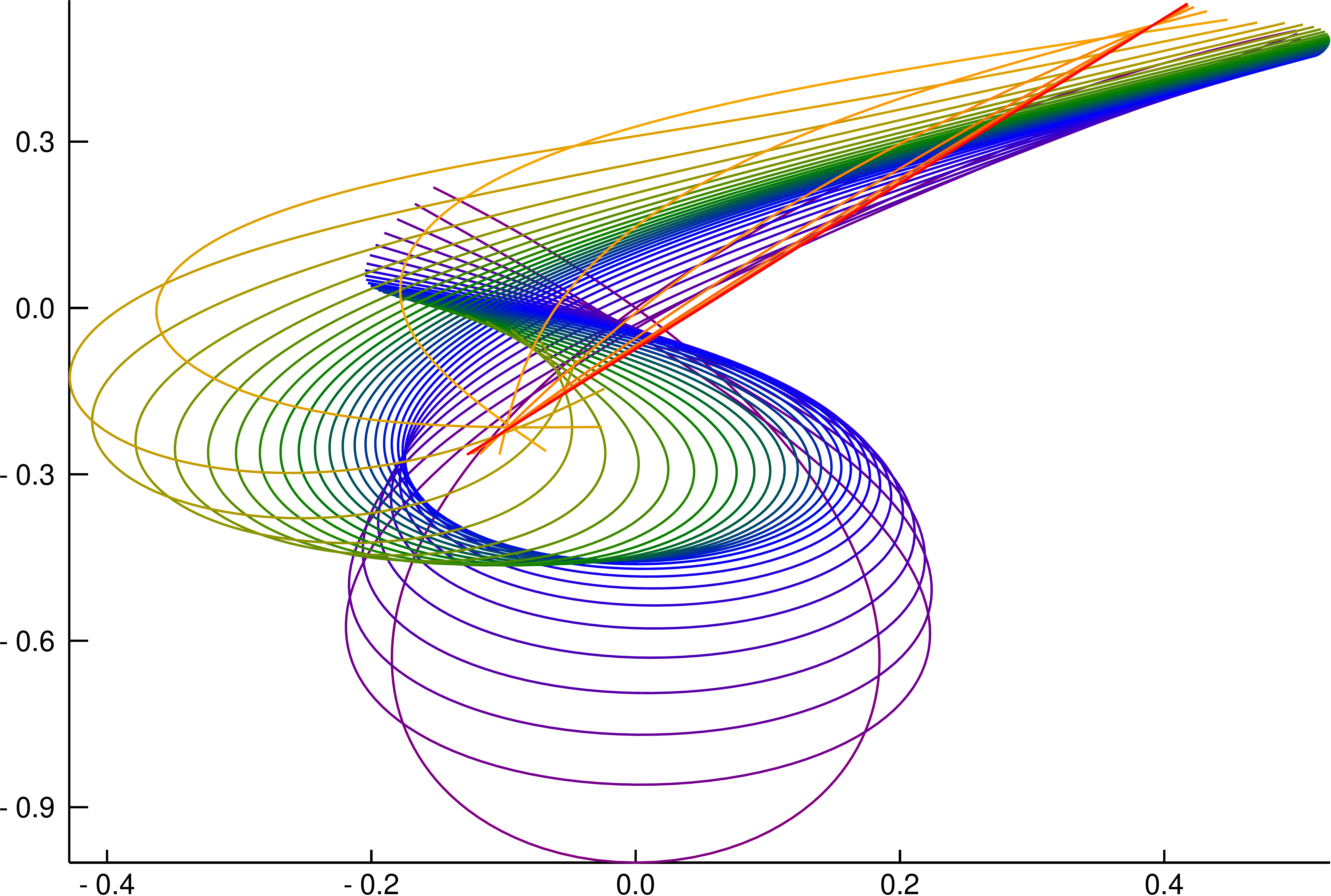}
		\caption{}
		\label{fig:asym_gamma_full}
	\end{subfigure}
	\begin{subfigure}{.49\textwidth}
		\includegraphics[width=\textwidth]{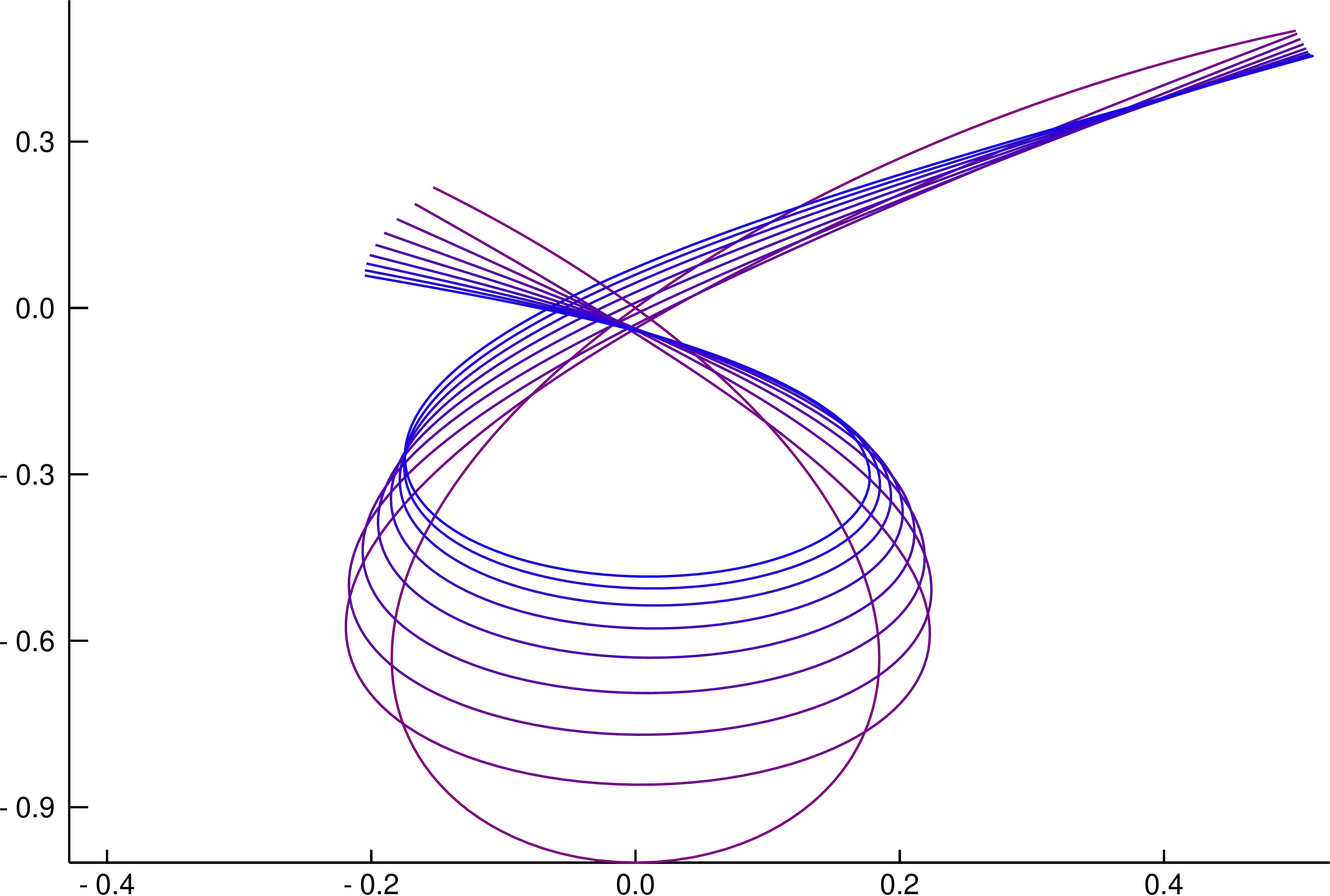}
		\caption{phase 1}
		\label{fig:asym_gamma_1}
	\end{subfigure}
	\begin{subfigure}{.49\textwidth}
		\includegraphics[width=\textwidth]{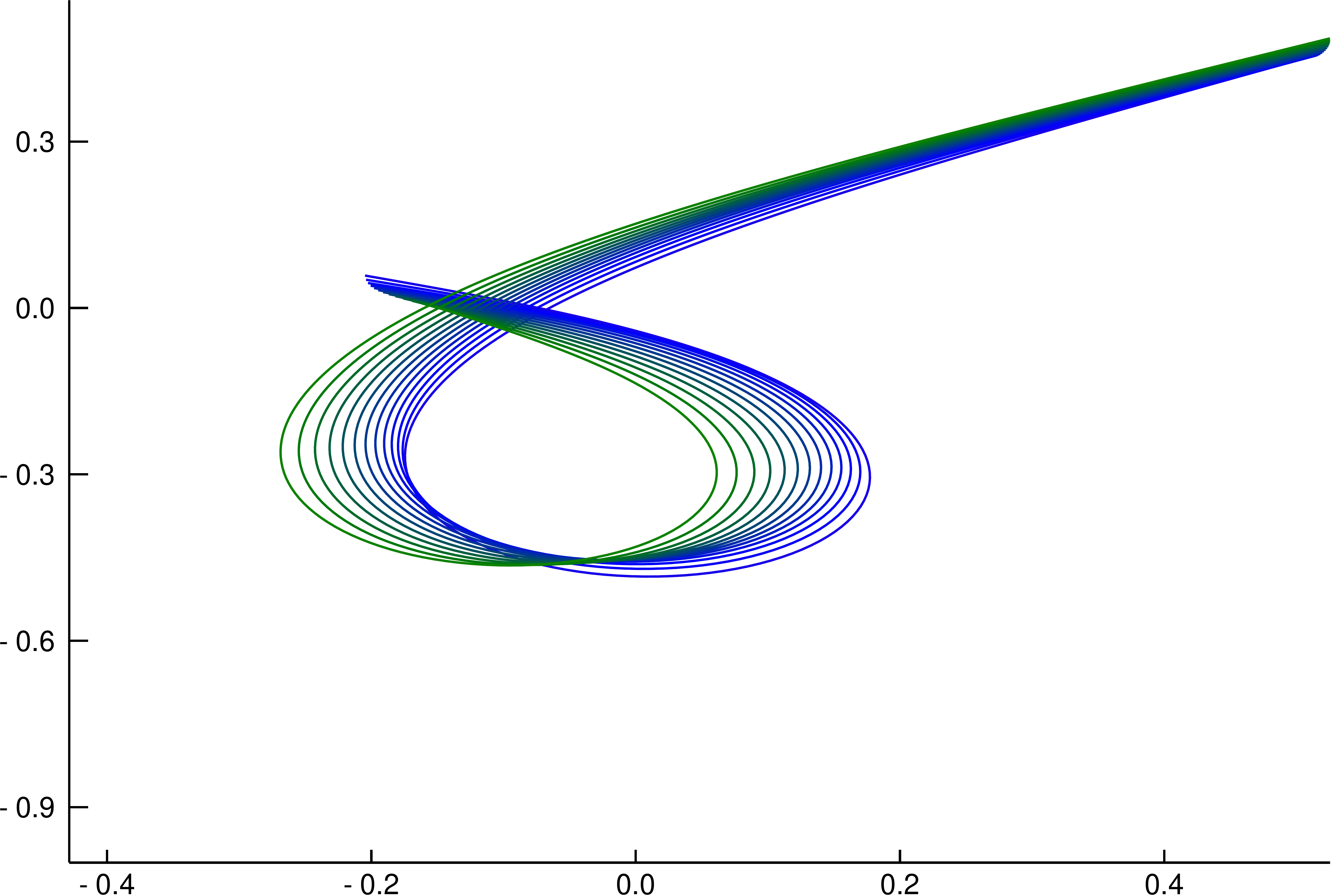}
		\caption{phase 2}
		\label{fig:asym_gamma_2}
	\end{subfigure}
	\begin{subfigure}{.49\textwidth}
		\includegraphics[width=\textwidth]{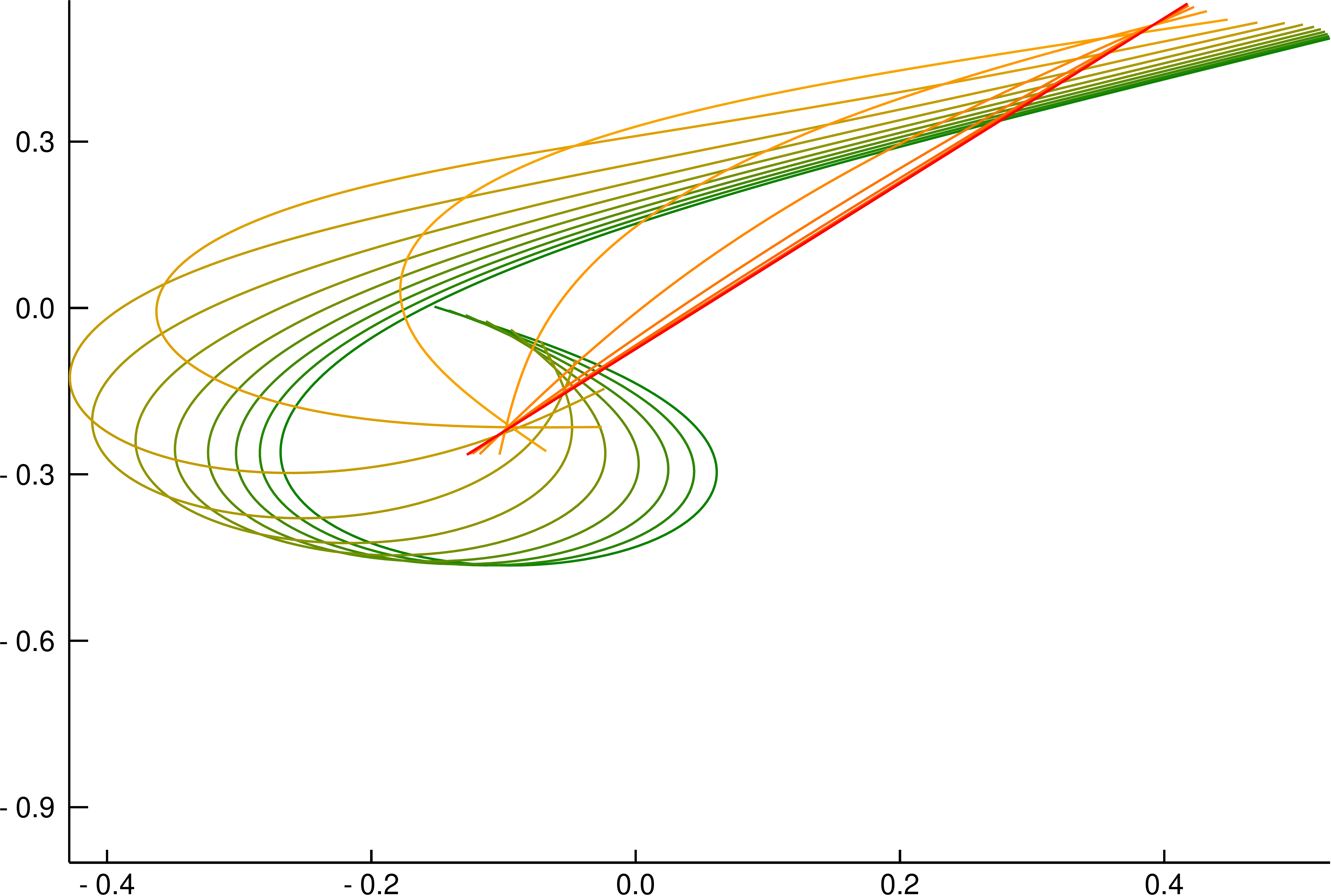}
		\caption{phase 3}
		\label{fig:asym_gamma_3}
	\end{subfigure}
	\caption{Evolution of an asymmetric $\curve$-shaped curve with $\tau = 0.01$ and $\varepsilon = 0.1$.}
	\label{fig:asym_gamma}
\end{figure}

\section{Acknowledgements}
This work was supported by the DFG Collaborative Research Center TRR 109, "Discretization in
Geometry and Dynamics". I wish to thank Marco Cicalese, Nicola Fusco and Carlo Mantegazza for many fruitful discussions.

\appendix
\section{Proof of Corollary \ref{cor:H4reg}}
We present in this appendix the remaining details for the proof of Corollary \ref{cor:H4reg}. Testing \eqref{Ann} with $\oeta(s, t) := \ophi(s) \opsi(t)$ with $\ophi \in C^\infty_c((0, 1); \RR^2)$ and $\opsi \in C^\infty_c(\RR_+)$ and using the arbitrariness of $\opsi$ we derive that for a.e. $t \in \RR_+$ it holds that
\begin{equation} \label{ELcont}
	\int_0^1 \frac{\varepsilon}{\len^3} \braket{\curve_{ss}}{\ophi_{ss}} + \frac{1}{\len} (1 - \frac{3\varepsilon}{2} \curva^2) \braket{\curve_s}{\ophi_s} + \frac{1}{\len} \braket{\velo}{\Rot(\curve_s)} \braket{\Rot(\curve_s)}{\ophi} \ds = 0.
\end{equation}
Integrating by parts in \eqref{ELcont} and employing the notation from the proof of Lemma \ref{Cordelia} then leads to
\begin{equation*}
	\int_0^1 \braket*{\frac{\varepsilon}{\len^3} \curve_{ss} - D_s^{-1} \left\{ \frac{1}{\len} (1 - \frac{3\varepsilon}{2} \curva^2) \curve_s \right\} + D_{s}^{-2} \left\{ \frac{1}{\len} \braket{\velo}{\Rot \curve_s} \Rot(\curve_s) \right\}}{\ophi_{ss}} \ds = 0
\end{equation*}
for a.e. $t \in \RR_+$. Hence, for a.e. $t \in \RR_+$ there exist $v(t)$, $w(t) \in \RR^2$ such that for all such $t$ and a.e. $s$ it holds that
\begin{equation*}
	-\frac{\varepsilon}{\len^3} \curve_{ss} = v + w s - D_s^{-1} \left\{ \frac{1}{\len} (1 - \frac{3\varepsilon}{2} \curva^2) \curve_s \right\} + D_{s}^{-2} \left\{ \frac{1}{\len} \braket{\velo}{\Rot(\curve_s)} \Rot(\curve_s) \right\}.
\end{equation*}
We then differentiate twice the equation above, which leads to
\begin{equation} \label{ELcont2}
	- \frac{\varepsilon}{\len^3} \curve_{ssss} = 3 \varepsilon \curva \curva_s \curve_s - \frac{1}{\len} (1 - \frac{3 \varepsilon}{2} \curva^2) \curve_{ss} + \frac{1}{\len} \braket{\velo}{\Rot(\curve_s)} \Rot(\curve_s),
\end{equation}
again for a.e. $t$ and $s$. By Young's inequality, \eqref{ELcont2}, \eqref{Bessie} and \eqref{Grace} we have for a.e. $t \in \RR_+$
\begin{align*}
	\norm{\curve_{ssss}}_{L^2}^2 &\leq \C(\varepsilon) \int_0^1 \abs{\curva}^2 \abs{\curva_s}^2 + \abs{\curve_{ss}}^2 + \abs{\curva}^4 \abs{\curve_{ss}}^2 + \abs{\velo}^2 \ds \\
	&\leq \C(\varepsilon) \int_0^1 \abs{\curve_{ss}}^2(\abs{\curve_{sss}}^2 + 1) + \abs{\curve_{ss}}^6 + \abs{\velo}^2 \ds \\
	&\leq \C(\varepsilon) \int_0^1 1 + \abs{\curve_{ss}}^6 + \abs{\curve_{sss}}^3 + \abs{\velo}^2 \ds.
\end{align*}
Furthermore by interpolation, Young's inequality with arbitrary $\delta > 0$ and \eqref{Grace} we have
\begin{align*}
	\norm{\curve_{sss}}_{L^3}^3 &\leq \C \left( \norm{\curve_{ssss}}_{L^2}^{\frac{21}{12}} \norm{\curve_{ss}}_{L^2}^{\frac{5}{4}} + \norm{\curve_{ss}}_{L^2}^3 \right) \\
	&\leq \C \left( \delta \norm{\curve_{ssss}}_{L^2}^2 + \C(\delta) \norm{\curve_{ss}}_{L^2}^{\frac{35}{12}} + \norm{\curve_{ss}}_{L^2}^3 \right) \\
	&\leq \C \delta \norm{\curve_{ssss}}_{L^2}^2 + \C(\delta, \varepsilon),
\end{align*}
as well as
\begin{align*}
	\norm{\curve_{ss}}_{L^6}^6 &\leq \C \left( \norm{\curve_{ssss}}_{L^2}^{1} \norm{\curve_{ss}}_{L^2}^{5} + \norm{\curve_{ss}}_{L^2}^6 \right) \\
	&\leq \C \left( \delta \norm{\curve_{ssss}}_{L^2}^2 + \C(\delta) \norm{\curve_{ss}}_{L^2}^{10} + \norm{\curve_{ss}}_{L^2}^6 \right) \\
	&\leq \C \delta \norm{\curve_{ssss}}_{L^2}^2 + \C(\delta, \varepsilon).
\end{align*}
Combining the above estimates leads to
\begin{equation*}
	\norm{\curve_{ssss}}_{L^2}^2 \leq \C(\varepsilon) \delta \norm{\curve_{ssss}}_{L^2}^2 + \C(\delta, \varepsilon) \int_0^1 1 + \abs{\velo}^2 \ds.
\end{equation*}
Hence chosing $\delta$ small enough we derive for a.e. $t \in \RR_+$
\begin{equation} \label{4sbound}
	\norm{\curve_{ssss}}_{L^2}^2 \leq \C(\varepsilon) \int_0^1 1 + \abs{\velo}^2 \ds.
\end{equation}
By integrating \eqref{4sbound} over $t \in [0, T]$ with arbitrary $T > 0$ and using \eqref{Georgia} we conclude the proof.

\clearpage
\bibliography{bib}
\bibliographystyle{plain}

\end{document}